\documentclass[12pt]{article}
\usepackage{pifont}
\usepackage{amsmath}
\usepackage{amssymb,float}
\usepackage{amsthm,graphicx,pifont}
\usepackage{graphicx,psfrag,float,subfigure}
\usepackage{caption,enumerate}
\usepackage[scale=0.8,a4paper]{geometry}
\usepackage[colorlinks]{hyperref}
\usepackage{amsthm,amsmath,amssymb}
\usepackage{mathrsfs}
\newtheorem{theorem}{Theorem}[section]

\newtheorem{lemma}[theorem]{Lemma}
\newtheorem{corollary}[theorem]{Corollary}

\title{On graphs with $1$-matching and $2$-matching edges}
\author{Yixuan Gao, Xiumei Wang, {Jinfeng Liu\thanks{Corresponding author: Jinfeng Liu. e-mail: ljf@zzu.edu.cn.}}\\
{\small  School of Mathematics and Statistics, Zhengzhou University}\\
{\small Zhengzhou 450001, China}}
\date{}\makeatother
\begin{document}
\maketitle

\begin{abstract}
Let \(G\) be a graph admitting a perfect matching. An edge is called a {\it \(k\)-matching edge} if it belongs to exactly \(k\) perfect matchings, and a {\it \(k^{+}\)-matching edge} if it belongs to at least \(k\) perfect matchings. Thus, {\it an admissible edge} is a \(1^{+}\)-matching edge, and a connected graph is {\it matching covered} if every edge is admissible. We call a connected graph {\it \(k\)-matching covered} if every edge is a \(k\)-matching edge; in particular, a \(2\)-matching covered graph is called {\it matching double covered}.

Motivated by matching-covered graph theory and the Berge--Fulkerson conjecture (1970s), we introduce the class \(\mathfrak{B}\) of connected graphs in which every edge is either a \(1\)-matching edge or a \(2\)-matching edge, and no perfect matching contains edges of both types. In particular, every matching double covered graph belongs to \(\mathfrak{B}\). Using ear decompositions and tight-cut decompositions, we establish a complete structural characterization of graphs in \(\mathfrak{B}\). These characterizations reveal how restrictions on the number of perfect matchings containing each edge determine the global structure of the corresponding matching-covered graphs.
\end{abstract}


{\bf 2000 MR Subject Classification} \ 05C70, 05C75

\section{Introduction}
Graphs considered in this paper are loopless, but they may have parallel edges. For standard graph-theoretical notation and terminology, we refer the readers to \cite{BM08} and \cite{LP86}. A {\it matching} in a graph $G$ is a set of pairwise nonadjacent edges. A {\it perfect matching} of $G$ is a matching which covers every vertex of $G$. An edge $e$ of $G$ is {\it admissible} in $G$ if there is a perfect matching $M$ of $G$ such that $e\in M$. An edge $e$ of $G$ is a {\it $k$-matching edge} if it is contained in exactly $k$ perfect matchings. If $e$ is contained by at least $k$ perfect matchings, we call $e$ a {\it $k^+$-matching edge}.  A connected graph $G$ is {\it matching covered} if each of its edges is admissible. A connected graph $G$ is {\it $k$-matching covered} if each of its edges is a $k$-matching edge. In particular, a $2$-matching covered graph is also called a {\it matching double covered} graph.



Let $G$ be a graph with vertex set $V(G)$. For any set $X$ of $V(G)$, we denote $C=\nabla_G(X)$ (or simply $C=\nabla(X))$ the cut of $G$ with $X$ and $\overline{X}=V(G)-X$    as its {\it shores}; in other words, $\nabla(X)$ is the set of all edges of $G$ which have precisely one end in $X$. Write $|C|=|\nabla(X)|$. A cut is {\it trivial} if either of its shores is a singleton. A cut is {\it odd (even)} if both its shores have an odd (even) number of vertices. The graph obtained from $G$ by contracting $\overline{X}$ to a single vertex $\overline{x}$ is denoted by $G\{X;\overline{x}\}$ and the graph obtained from $G$ by contracting $X$ to a single vertex $x$ is denoted by $G\{\overline{X};x\}$; We call $ x$ and $\overline{x}$ the {\it contraction vertices of $C$}. We shall refer to these
two graphs $G\{X;\overline{x}\}$ and $G\{\overline{X};x\}$ as the {\it $C$-contractions of $G$}.

Let $G$ be a matching covered graph and let $X$ be an odd subset of $V(G)$. Then, the
cut $C:=\nabla(X)$ is called a {\it separating cut} of $G$ if the two $C$-contractions of $G$ are matching covered. A cut $C:=\nabla(X)$ is a {\it tight cut} of $G$ if $|C\cap M|=1$ for every  perfect matching $M$ of $G$. Clearly, every trivial cut is a tight cut. It is easy to check that every tight cut of $G$ is a separating cut of $G$. Thus if $C$ is a nontrivial tight cut of $G$, the two $C$-contractions of $G$ are matching
covered graphs that have fewer vertices than $G$. If either of the $C$-contractions has a nontrivial tight cut, that graph can be further decomposed into even smaller matching covered graphs. This procedure can be repeated until one obtains a
list of graphs each of which is a matching covered graph that has no nontrivial tight cuts. This procedure is known as {\it a tight cut decomposition} of $G$.

A matching covered graph which is free of nontrivial tight cuts is a {\it brace} if it is bipartite, and a {\it brick} if it is nonbipartite. Lov\'{a}sz proved the following remarkable result:

\begin{lemma}\cite{L1986}\label{L1986}
The results of any two applications of the tight cut decomposition procedure on a matching covered graph $G$ are the same list of bricks and braces except possibly for the multiplicities of edges.
\end{lemma}

A cycle is {\it even} if it has an even number of edges, and  {\it odd} otherwise. A cycle is called a $k$-cycle if it contains $k$ edges.  A path is {\it even} if it has an even number of edges, that is, its length is even, and  {\it odd} otherwise.  A {\it single ear} of a graph $G$ is a odd path whose internal vertices (if any) all have degree two in $G$. A {\it double ear} of $G$
is a pair of vertex-disjoint single ears of $G$. An {\it ear} of $G$ is either a single ear
or a double ear of $G$. An {\it ear decomposition} of a matching covered graph $G$ is a sequence $(G_1, G_2, \ldots, G_r)$ of matching covered subgraphs of $G$ such that: (i) $G_1=K_2$, (ii) $G_{i+1}=G_i\cup P_i$, where $P_i$ is an ear (single or double) of
$G_{i}$ in $G$, where $1\leq i<r$, (iii) $G_r=G$.  In this case, we can call {\it $G$ has an ear decomposition such that $G=K_2+P_1+\cdots+P_r$}, where $P_1,P_2,\ldots,P_r$ are the successive ears used in the ear decomposition of $G$. The following fundamental results on ear decompositions were established by Lov\'{a}sz and Plummer.

\begin{lemma}\cite{LP86}
Every matching covered graph has an ear decomposition.
\end{lemma}

\begin{lemma}\cite{LP86}\label{L-ear-decom-bipartite}
Given any bipartite matching covered graph $G$, there exists a sequence
$$G_1\subset G_2 \subset \cdots \subset G_r$$
of matching covered subgraphs of $G$, where (i) $G_1=K_2$, $G_r=G$ and (ii) $G_{i+1}$ is
the union of $G_{i}$ and a single ear of $G_{i+1}$, $1\leq i\leq r-1$.
\end{lemma}

\begin{lemma}\cite{X2018}\label{uni-pm}
Let $G$ be a connected and simple graph. If each edge of $G$ belongs to a unique perfect matching and $V(G)\geq 6$, then $G$ is an even cycle.
\end{lemma}

Suppose that $G$ is a connected matching covered graph and each of its edges is in a unique perfect matching. If $|V(G)|=2$, then $G=K_{2}^*$ which is $K_{2}$ with parallel edges; if $|V(G)|=4$, then $G=K_{4}$ or 4-cycle. Combining this fact with Lemma \ref{uni-pm}, we can obtain:

\begin{lemma}\label{uniqe}
If $G$ is a connected matching covered graph in which each edge is contained in exactly one perfect matching, then $G$ is $K_{2}^*$, $K_{4}$ or an even cycle $C_{2n}$, where $K_{2}^*$ is the graph obtained from $K_{2}$ by adding parallel edges.
\end{lemma}

There is a great amount of literature on graphs whose edges are in some maximum matching or perfect matching. A graph is {\it equimatchable} in which every matching can be extended to a maximum matching. Lesk {\it et al}. \cite{ML84} obtained a structural characterization of equimatchable graphs via the Gallai-Edmonds Structure Theorem, and yielded a polynomial-time recognition algorithm. Akbari {\it et al}. \cite{SA182} obtained a characterization of claw-free equimatchable graphs. Akbari {\it et al}. \cite{SA183} showed  that if $G$ is a equimatchable $r$-regular graph for an odd positive integer $r$, then $G$ is $K_{r+1}$ or $K_{r,r}$. In the same paper, they also proved that the only triangle-free equimatchable $r$-regular graphs are $C_{5}$, $C_{7}$ and $K_{r,r}$, where $r$ is a positive integer. B$\rm \ddot{u}$y$\rm \ddot{u}$k\c{c}olak {\it et al}. \cite{YB22} obtained a complete characterization of triangle-free equimatchable graphs and a linear time algorithm that recognizes whether a given nonbipartite graph is equimatchable and triangle-free. The class of equimatchable graphs having perfect matchings are called {\it randomly matchable graphs}. Sumner \cite{DP79} showed that a graph is randomly matchable if and only if each connected component is isomorphic to $K_{2n}$ or $K_{n,n}$ for $n\geq 1$. Mkrchyan \cite{VV06} proved that minimal graphs in which every edge can be extended to a maximum matching without isolated vertices contain a perfect matching. Niu {\it et al}. \cite{NZWL2024} obtained characterizations of factor-critical graphs and bipartite graphs with positive surplus each of whose edges belongs to at most one maximum matching. In the same paper, they also presented the structure of graphs each of whose edges  belongs to at most one perfect matching. In this paper, we will be considered with the new class $\mathfrak{B}$: graphs for which very edge is in exactly one perfect matching or two perfect matchings, and a $1$-matching edge and $2$-matching edge can not in a common perfect matching. Obviously, a matching double covered graph is in $\mathfrak{B}$. In this paper, we focus on characterizing the structures of matching double covered graphs and graphs in $\mathfrak{B}$ respectively. The paper is organized as follows: In Section 2, we first focus on bipartite graphs in the set $\mathfrak{B}$ and obtain:

\begin{theorem}\label{brace-B}
If $G$ is a bipartite matching covered graph, then $G$ is  in $\mathfrak{B}$ if and only if $G$ is $K_{2}^*$, $C_{4}^2$, $C_{4}^4$, $K_{3,3}$ or $C_{2n}$ (see Figure \ref{T6}).
\end{theorem}

\begin{figure}[h!]
  \centering
  \includegraphics[width=5.5 in]{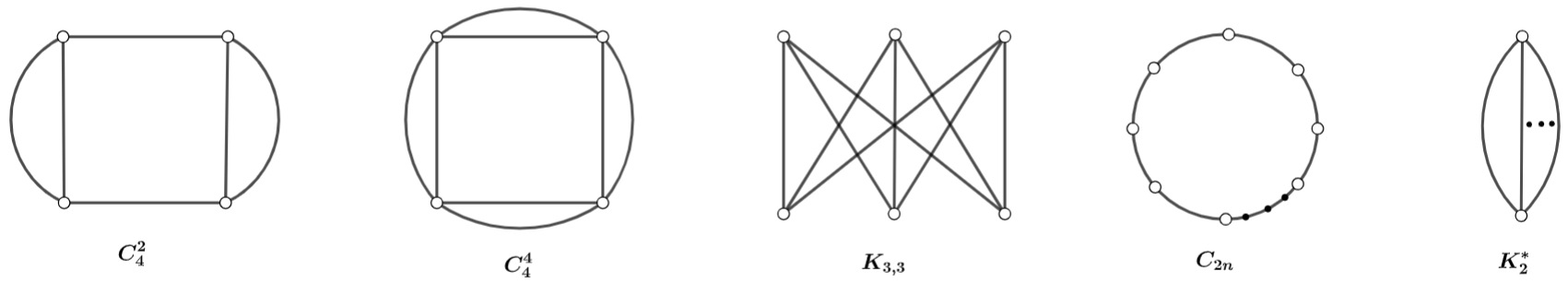}
    \vspace{-0.3cm}
  \caption{The bipartite graphs in $\mathfrak{B}$}\label{T6}
\end{figure}

Because $K_{3,3}$ and $C_{4}^4$ are both matching double covered and braces, we can get conclusions as follows.

\begin{corollary}\label{TH2.1}
Let $G$ be a bipartite graph. Then $G$ is matching double covered if and only if $G$ is $K_{3,3}$ or $C_{4}^4$, see Figure \ref{T6}.
\end{corollary}

\begin{corollary}\label{TH1}
Let $G$ be a bipartite graph. Then $G$ is matching double covered brace if and only if it is $K_{3,3}$ or $C_{4}^4$.
\end{corollary}

In Section 4, we continue to research nonbipartite graphs in matching double covered graphs and the set $\mathfrak{B}$ respectively. We firstly obtain the structure of bricks in $\mathfrak{B}$ by presenting the following result.

\begin{theorem}\label{brick-B}
Let $G$ be a brick. Then $G$ is in $\mathfrak{B}$ if and only if $G$ is $K_{4}$, $K_{4}^2$, $K_{4}^4$, the Petersen graph $P$, $T_{1}$ or $K_{4}^6$ (see Figure \ref{T14}).
\end{theorem}

\begin{figure}[h!]
  \centering
  \includegraphics[width=4.0 in]{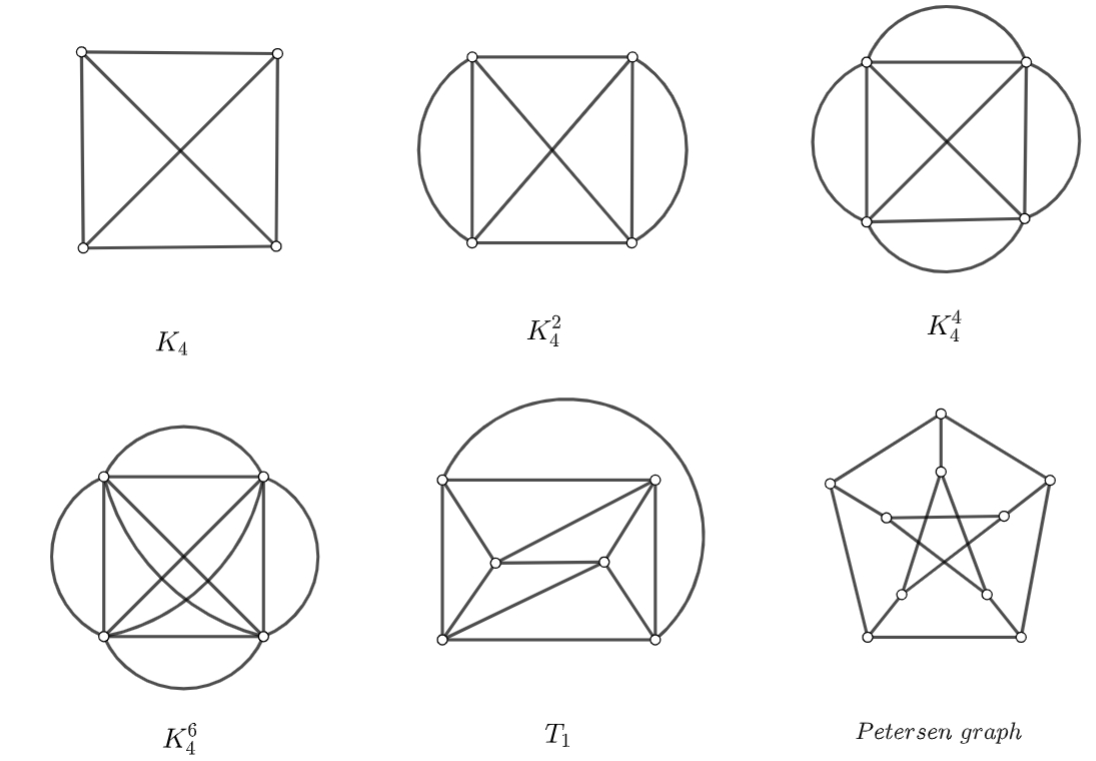}
  \caption{The bricks in $\mathfrak{B}$}\label{T14}
\end{figure}

\begin{figure}[h!]
  \centering
  \includegraphics[width=4.5 in]{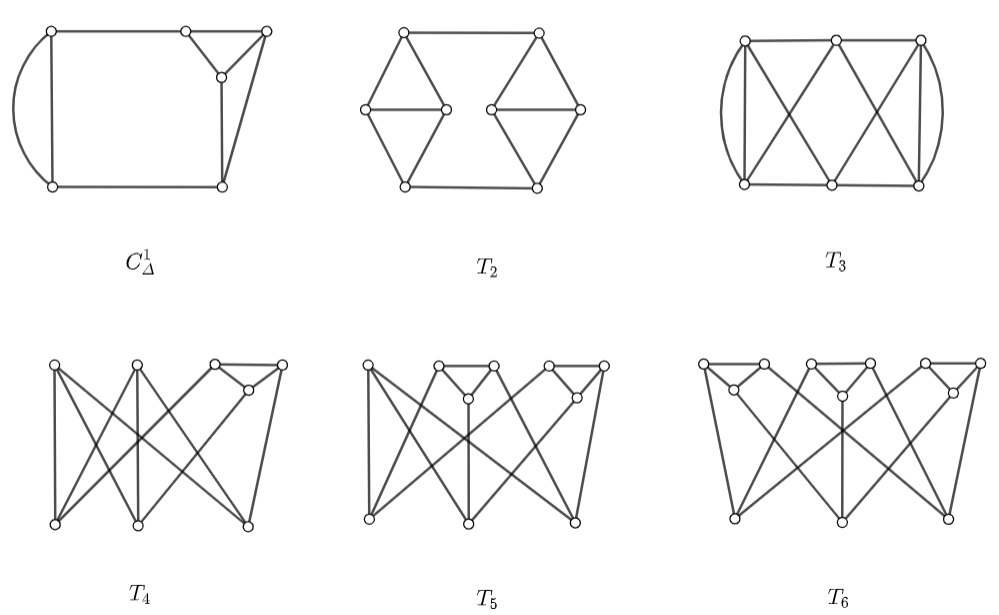}
  \caption{Nonbipartite matching covered graph with a nontrivial cut in $\mathfrak{B}$}\label{$T_{5}-T_{8}$}
\end{figure}

Note that the Petersen graph $P$, $T_{1}$ and $K_{4}^6$ are matching double covered. Then as a corollary of Theorem \ref{brick-B}, we can obtain the following result.

\begin{corollary}\label{TH3.2}
Let $G$ be a brick. Then $G$ is matching double covered if and only if it is the Petersen graph, $T_{1}$ or $K_{4}^6$ (see Figure \ref{T14}).
\end{corollary}

Lately in Section 4, we focus on characterizing the matching covered graphs with nontrivial tight cuts in matching double covered graphs and the set $\mathfrak{B}$ respectively, and achieve:

\begin{theorem}\label{TH3}
Let $G$ be a nonbipartite matching covered graph with a nontrivial cut. Then $G$ is in $\mathfrak{B}$ if and only if $G$ is $C_{\bigtriangleup}^1$, $T_2$, $T_3$, $T_4$, $T_5$ or $T_6$ (see Figure \ref{$T_{5}-T_{8}$}).
\end{theorem}

As an immediate corollary, we can obtain the following result since $T_3$, $T_4$, $T_5$ and $T_6$ are matching double covered.

\begin{corollary}\label{cor3}
Let $G$ be a nonbipartite matching covered graph with a nontrivial cut. Then $G$ is matching double covered if and only if $G$ is $T_3$, $T_4$, $T_5$ or $T_6$ (see Figure \ref{$T_{5}-T_{8}$}).
\end{corollary}

\section{Bipartite matching covered graphs}

All the graphs discussed in this section are bipartite. We will give the proof of Theorem \ref{brace-B} characterizing the structures of graphs in the set $\mathfrak{B}$. Before proceeding, we first present a property of graphs in $\mathfrak{B}$.

\begin{lemma}\label{B-regular}
If $G\in \mathfrak{B}$, then $G$ is regular.
\end{lemma}

\begin{proof}
Suppose $G$ is not regular. We can choose $v_1$, $v_2\in V(G)$ such that $d(v_1)>d(v_2)$, where $d(v_1)$ and $d(v_2)$ are degrees of $v_1$ and $v_2$ respectively. Let $\Phi(G)$ be the number of perfect matchings in $G$. Since $G\in \mathfrak{B}$, each edge in $G$ is either a $1$-matching edge or a $2$-matching edge. Hence, $d(v_2)\leq \Phi(G)\leq d(v_1)$. Set $\Phi(G)=d(v_1)+i$ ($0\leq i\leq d(v_1)$).

By the structure of $G$, we can obtain that there exist $i$ $2$-matching edges and $d(v_1)-i$ $1$-matching edges incident with $v_1$, and there exist $d(v_1)+i-d(v_2)$ $2$-matching edges and $2d(v_2)-d(v_1)-i$ $1$-matching edges incident with $v_2$. Here $d(v_1)+i-d(v_2)>i$, which implies the number of $2$-matching edges incident with $v_2$ is strictly larger than the number of $2$-matching edges incident with $v_1$. But this contradicts the fact that $G\in \mathfrak{B}$. Hence $G$ is regular.
\end{proof}

Let $G$ be a bipartite graph. If $|V(G)|=2$,
then it is $K_{2}$ or $K_{2}^*$. Every edge of $K_{2}$ or $K_{2}^*$ is $1$-matching edge.
Therefore, we just consider the graphs with at least four vertices in the sequel. Notice that a graph $G$ is not in $\mathfrak{B}$ if $G$ contains $3^+$-matching edges. Hence we can get three lemmas below which will be used to prove Theorem \ref{TH2.1}.

\begin{lemma}\label{C-3+}
If there exist at least three parallel edges between two adjacent vertices in $C_{2n}^*$ ($n\geq 2$), then $C_{2n}^*$ contains $3^+$-matching edges.
\end{lemma}

\begin{proof}
Let $x$ and $y$ be two adjacent vertices of $C_{2n}^*$, and $e_1, e_2, \ldots, e_t$ ($t\geq 3$) be the corresponding parallel edges. Since $n\geq 2$, $C_{2n}^*-x-y$ has a perfect matching $M$. Thus $M\cup e_i$, $i=1, 2, \ldots, t$, are $t$ perfect matchings of $C_{2n}^*$. Hence every edge in $M$ is a $3^+$-matching edge of $C_{2n}^*$.
\end{proof}

\begin{lemma}\label{L-two-cycles}
Let $G$ be a matching double covered graph. If $C_{1}$ and $C_{2}$ are two vertex-disjoint even cycles of $G$, then $G-V(C_{1}\cup C_{2})$ has no perfect matchings.
\end{lemma}

\begin{proof}
By contradiction and suppose that $G-V(C_{1}\cup C_{2})$ has a perfect matching $M$. Then we can check that every edge in $M$ is a $3^+$-matching edge of $G$, a contradiction.
\end{proof}

From the process of an ear decomposition, we can get a property of matching covered graphs as follows.

\begin{lemma}\label{L-sub-DMC}
Let $G$ be a bipartite matching covered graph, $G'$ be a subgraph of $G$ during the process of an ear decomposition. If there exists a $3^+$-matching edge $e$ in $G'$, then $e$ is also a $3^+$-matching edge in $G$.
\end{lemma}

\begin{proof}
Since $e$ is a $3^+$-matching edge $e$ in $G'$, there exist three perfect matchings $M_{1}$, $M_{2}$ and $M_{3}$ containing $e$. By Lemma \ref{L-ear-decom-bipartite}, $G'$ is a nice subgraph of $G$, there exists a perfect matching $M$ of $G-V(G')$. Thus $M_{1} \cup M$, $M_{2} \cup M$ and $M_{3} \cup M$ are three perfect matchings of $G$ containing $e$, which implies that $e$ is a $3^+$-matching edge in $G$.
\end{proof}

Let $P=v_1e_1v_2\cdots v_ke_{k-1}v_{k}$ be a odd path. The vertices $v_1$ and $v_k$ are called the {\it ends} of $P$. Call an edge $e_i$ is an {\it even or odd edge} of a odd $P$ if $i$ is even or odd. Denote by $N(P)$ and $O(P)$ the set of even edges and odd edges of odd path $P$ respectively. Let $x$ and $y$ be two vertices of a path $P$, we denote by $xPy$ the segment of $P$ from $x$ to $y$. The notation $xPy$ is also used simply to signify a $xy$-path $P$. For a cycle $C$ with $u,v\in V(C)$, we write $uCv$ the segment of $C$ from $u$ to $v$ clockwise.

\begin{lemma}\label{L-two-ears}
Let $G$ be a bipartite matching covered graph which is not a cycle and $|V(G)|\geq4$, there exists an ear decomposition of $G$ such that $G=C+P_1+\cdots+P_r$, $P_{1}$ is a nontrivial ear with ends $x$ and $y$. The following hold:
\vspace{-9pt}
\begin{enumerate}[(\romannumeral1)]
\setlength{\itemsep}{-1ex}
\item the graph $C+P_{1}$ has three different perfect matchings, and the even edges of $P_1$ are $2$-matching edges in $C+P_1$.

If $G$ contains no $3^+$-matching edges, then

\item $P_{2},\ldots,P_{n}$ are all trivial ears with ends are not $x$ and $y$.

\item the ends of $P_{i}$ $(2\leq i\leq n)$ can not both lie in $xCy$ or $yCx$. Moreover, if there exists an ear $P_{i}$ $(2\leq i\leq n)$ with ends lie in $P_{1}$, then $C$ is a $2$-cycle.
\end{enumerate}
\end{lemma}

\begin{proof}
By Lemma \ref{L-ear-decom-bipartite}, $G$ has an ear decomposition such that $G=K_2+Q_1+\cdots+Q_{r+1}$, where $Q_i$ are ears of $G$. Then $K_2+Q_1$ is an even cycle, denoted by $C$. Note that $|V(G)|\geq 4$. If $|V(C)|=2$, then there exists a nontrivial ear $Q_j$ $(2\leq j\leq r+1)$. Set $P_1=Q_j$, $P_i=Q_i$ for $2\leq i\leq j-1$, $P_k=Q_{k+1}$ for $j\leq k\leq r$. Then $G$ has an ear decomposition such that $G=C+P_1+P_2\cdots+P_r$, and $P_1$ is a nontrivial ear. If $|V(C)|\geq 4$, the ends of $Q_{1}$ are $x$ and $y$, then $C$ can be divided into two odd paths $xCy$ and $yCx$ by $x$ and $y$, and one of these paths has length at least three. Without loss of generality, suppose that the length of $xCy$ is at least three. Change the cycle $K_2+Q_1$ into $K_2+Q_2$. Set $P_1=Q_1$, $P_i=Q_{i+1}$ for $2\leq i\leq r$. Thus the ear decomposition of $G=C+P_1+P_2+\cdots+P_r$ is the desired ear decomposition, where $P_1$ is a nontrivial ear.

In $C+P_1$, the odd edges of $xCy$, and the even edges of $yCx$ and $P_{1}$ consist a perfect matching $M_{1}$; the even edges of $xCy$ and $P_{1}$, the odd edges of $yCx$ consist a perfect matching $M_{2}$; the even edges of $xCy$ and $yCx$, the odd edges of $P_{1}$ consist a perfect matching $M_{3}$. Therefore $C+P_{1}$ has three different perfect matchings, and the even edges of $P_1$ are $2$-matching edges in $C+P_1$. Hence (i) holds.

In the sequel we consider the case that $G$ contains no $3^+$-matching edges. Let $G_{i}=C+P_{1}+\cdots+P_{i}$, where $1\leq i\leq r$. Suppose that there exist nontrivial ears in $P_{2},\ldots,P_{n}$ and $P_{i}$ $(2\leq i\leq n)$ is the first nontrivial ear. Recall that $N(P_{i})$ is the set of even edges of $P_{i}$. By the analysis of $(i)$, $G_1=C+P_1$ has three different perfect matchings $M_{1}$, $M_{2}$, and $M_{3}$. Then $N(P_{i})\cup M_{1}$, $N(P_{i})\cup M_{2}$, $N(P_{i})\cup M_{3}$ are three different perfect matchings of $G_{i}$, therefore the edges of $N(P_{i})$ are $3^+$-matching edges in $G_{i}$ and thus are $3^+$-matching edges in $G$ by Lemma \ref{L-sub-DMC}, a contradiction. Hence $P_{2},\ldots,P_{n}$ are all trivial ears.

Suppose that there exists a ear $P_j$ $(2\leq j\leq n)$ with ends are $x$ and $y$, then the perfect matchings of $G_{1}=C+P_{1}$ are also the perfect matchings of $G_{j}$ $(j>1)$. Since $G_j$ is matching covered, there exists a perfect matching containing $P_{j}$ and the even edges of $P_{1}$ in $G_j$. By $(i)$, the even edges of $P_{1}$ are $2$-matching edges in $G_{1}$, hence these edges are $3^+$-matching edges in $G_{i}$, which implies that $G$ has a $3^+$-matching edge by Lemma \ref{L-sub-DMC}, a contradiction. Thus $P_{2},\ldots,P_{n}$ are all trivial ears with ends are not $x$ and $y$, which implies (ii) holds.

Let the ends of $P_{i}$ $(2\leq i\leq n)$ be $u$ and $v$. Suppose without loss of generality that $u$ and $x$ are in the same part, $v$ and $y$ are in the other part in $G$. Since $G$ contains no $3^+$-matching edges, $P_{2},\ldots,P_{n}$ are trivial ears with ends not $x$ and $y$ by (ii), the perfect matchings of $G_{1}=C+P_{1}$ are also the perfect matchings of $G_{j}=C+P_1+\cdots+P_j$ $(j>1)$. We first prove that the ends of $P_{i}$ $(2\leq i\leq n)$ can not lie in $xCy$ or $yCx$.  Without loss of generality, suppose that the ends of $P_{i}$ lie in $xCy$, then the length of $xCy$ is at least $3$. Thus $P_i=uv$, the even edges of $P_1$ , and a perfect matching of $C-\{u,v\}$ consist of a perfect matching $M'$ of $G_i$. Note that $M'$ is not a perfect matching of $G_1$ and the even edges of $P_1$ are $2$-matching in $G_1$ by (i). Hence the even edges of $P_1$ are $3^+$-matching edges in $G_i$, which deduces $G$ contains a $3^+$-matching edge by Lemma \ref{L-sub-DMC}, a contradiction.

Then we consider the case that the ends $u$ and $v$ of $P_{i}$ $(2\leq i\leq n)$ lie in $P_{1}$. Suppose by contradiction that $C$ is not a $2$-cycle. Then there is at least one path, say, $xCy$ with length is not less than three. If $v$ and $u$ appear successively on $xP_{1}y$, then $v$ and  $u$ are the inner vertices of $xP_{1}y$. Denote by $C_{i}$ the cycle $vP_{1}uP_{i}v$, then $G_{i}-V(C\cup C_{i})$ has a perfect matching $N$ (may be an empty set). Since $P_{i}$ is an edge in $C_{i}$, let $N_1$ be the perfect matching of $C_{i}$ containing $P_{i}$. Let $N_{2}$ be the perfect matching of $C$ containing the even edges of $xCy$. So $M''$=$N_{1}\cup N_{2}\cup N$ is a perfect matching of $G_{i}$ containing $P_{i}$ and even edges of $xCy$. If $u$ and $v$ appear successively on $xP_{1}y$ (at this time $u$ may be same as $x$, $v$ may be same as $y$, but $\{u,v\}\neq \{x,y\}$), then $P_{i}$, the even edges of $uP_{1}x$, $uP_{1}v$, $vP_{1}y$, $xCy$ and the odd edges of $yCx$ constitute a perfect matching $M'''$ of $G_{i}$. One can check that $M''$ and $M'''$ are not perfect matchings of $G_1$ and both contain the even edges of $xCy$. Together with the fact that the even edges of $xCy$ are $2$-matching in $G_1$, we obtain the even edges of of $xCy$ are $3^+$-matching edges in $G_i$, which deduces $G$ contains a $3^+$-matching edge by Lemma \ref{L-sub-DMC}, a contradiction. Hence we get $C$ is a $2$-cycle. Thus (iii) is proved.
\end{proof}


\begin{lemma}\label{c2-p2}
Let $G$ be a bipartite matching covered graph with no $3^+$-matching edges. Suppose that $G$ has an ear decomposition $G=C+P_1+\cdots+P_r$ such that $C$ is a 2-cycle, $P_1$ is a nontrivial ear, $P_i$ are trivial ears, $2\leq i\leq n$. If the length of $P_1$ is at least  $5$, then the ends of $P_i$ ($2\leq i\leq r$) are same to the ends of some odd edge of $P_1$.
\end{lemma}

\begin{proof}
Let $x$ and $y$ be the ends of $P_1$, $u$ and $v$ be the ends of $P_i$ ($2\leq i\leq r$), where $u$ and $x$ are in the same part. Set $G_{i}=C+P_{1}+\cdots+P_{i}$ $(1\leq i\leq n)$.
From Lemma \ref{L-two-ears}, we know that $\{x,y\}\neq \{u,v\}$. Since $C$ is a $2$-cycle, and $P_{1}$ is a nontrivial ear, the ends $u$ and $v$ of $P_{i}$ falls on $P_{1}$ at the same time.

Suppose $u$ and $v$ are not adjacent in $P_{1}$. If $v$ and $u$ appear successively in $xP_{1}y$, then $v$ and $u$ are inner vertices of $P_{1}$. Denote by $C_{i}$ the cycle $vP_{1}uP_{i}v$, then $G_{i}-V(C\cup C_{i})$ has a perfect matching $N$. Let $N_{1}$ be the perfect matching of $C_{i}$ containing $P_{i}$ and the even edges of $vP_{1}u$. There are two disjoint perfect matchings $N_{2}$ and $N_{3}$ of $C$, so $N_{1}\cup N_{2}\cup N$ and $N_{1}\cup N_{3}\cup N$ are two perfect matchings of $G_{i}$ containing $P_{i}$ and the even edges $vP_{1}u$. Here the even edges of $vP_{1}u$, as the odd edges of $P_{1}$, are $1$-matching edges of $G_{1}$, thus are $3^+$-matching edges of $G_{i}$. Then by Lemma \ref{L-sub-DMC}, $G$ has $3^+$-matching edges, a contradiction. If $u$ and $v$ appear successively in $xP_{1}y$ ($u$ may be $x$, $v$ may be $y$, but $\{u,v\}\neq \{x,y\}$), then $P_{i}$, the even edges of $uP_{1}x$, $uP_{1}v$ and $vP_{1}y$ form a perfect matching of $G_{i}$. Here the even edges of $uP_{1}v$, as the even edges of $P_{1}$, are the $2$-matching edges of $G_{1}$, thus are $3^+$-matching edges of $G_{i}$, which implies $G$ contains $3^+$-matching edges by Lemma \ref{L-sub-DMC}, a contradiction. Therefore, $u$ and $v$ are two adjacent vertices in $P_{1}$.

Suppose now that $u$ and $v$ are the ends of an even edge $e$ of $P_{1}$. Let  $C_{1}$ be the cycle composed of $P_{i}$ and this even edge. Since $|P_{1}|\geq 5$, $G_{i}-V(C\cup C_{1})$ has a perfect matching $N$. By Lemma \ref{L-two-cycles}, the edges in $N$ are $3^+$-matching edges of $G_{i}$, which deduces that $G$ has $3^+$-matching edges by Lemma \ref{L-sub-DMC}, a contradict. Hence the ends $u$ and $v$ of $P_{i}$ are the ends of some odd edge of $P_{1}$. The result holds.
\end{proof}


Now we present the proof of Theorem \ref{brace-B} using the above results.

{\bf Proof of Theorem \ref{brace-B}.} By definition, one can check that $K_{2}^*$, $C_{4}^2$, $C_{4}^4$, $K_{3,3}$ or $C_{2n}$ are bipartite graphs in $\mathfrak{B}$. Hence the sufficiency follows. We now establish the necessity.  Suppose that $G$ is a bipartite graph in $\mathfrak{B}$. Then $G$ is a bipartite matching covered graph and contains no $3^+$-matching edges. Through these structural properties, one can check that then $G$ is $K_{2}^*$ if $|V(G)|=2$, and $G$ is $C_4$, $C_{4}^2$ or $C_{4}^4$ if $|V(G)|=4$. If $|V(G)|=2n\geq 6$, obviously, $G$ may be $C_{2n}$. If $G\ncong C_{2n}$, then by Lemma \ref{L-two-ears}, there exists an ear decomposition of $G$ such that $G=C+P_1+\cdots+P_r$, where $r\geq 1$ and $P_{1}$ is a nontrivial ear with ends $x$ and $y$, $P_2,\ldots, P_r$ are trivial ears with ends are not $\{x,y\}$. Let $G_i=C+P_1+\cdots+P_i$. By Lemma \ref{L-two-ears} (iii), the ends of $P_{i}$ $(2\leq i\leq n)$ can not both lie in $xCy$ or $yCx$. With regard to the positions of the ends of $P_2$, we consider the following two cases.

{\bf Case 1.} The ends of $P_ {2}$ are in $P_ {1}$. By Lemma \ref{L-two-ears} (iii), the $C$ is a $2$-cycle, $P_2, \ldots, P_r$ are trivial ears with ends in $P_1$. Since $|V(G)|\geq 6$, the length of $P_1$ is at least $5$. Hence by Lemma \ref{c2-p2}, the ends $u$ and $v$ of $P_{2}$ are the ends of a odd edge $e$ of $P_{1}$, i.e., $e=uv$. Note that $G_2$ is not regular, then $G_2\notin \mathfrak{B}$ by Lemma \ref{B-regular}. Hence $G\neq G_{2}$, $P_{3}$ exists. From Lemmas \ref{C-3+} and \ref{c2-p2}, the ends of $P_{3}$ are the ends of a odd edge $f$ of $P_{1}$ and $e\neq f$. Denote by $C_{1}$ the cycle formed by $P_{2}$ and $e$, and $C_{2}$ the cycle formed by $P_{3}$ and $f$. So $V(C_{1})\cap V(C_{2})=\emptyset$. Since $|V(G)|\geq 6$, $G_{3}-V(C_{1}\cup C_{2})$ has a perfect matching $N$. Then edges in $N$ are $3^+$-matching edges of $G_{3}$, so are $3^+$-matching edges of $G$  by Lemma \ref{L-sub-DMC}, a contradiction.

{\bf Case 2.} The ends of $P_{2}$ are in two of paths $P_{1}$, $xCy$, $yCx$ respectively. For the convenience of present the proof, we color the vertices in two parts of $G$ with black and white respectively such that $x$ has color black and $y$ has color white. Due to symmetry, we may as well suppose that the ends $u$ and $v$ of $P_{2}$ are in $P_{1}$ and $xCy$, say $u$ is black in $xCy$ and $v$ is white in $P_{1}$. Obviously $v$ is a inner vertex of $P_{1}$, and $u$ is a inner vertex of $xCy$. Since $x$ and $u$ are black, and $y$ and $v$ are white, the length of $xCu$ is at least $2$, and the length of $yP_{1}v$ is at least $2$. Suppose that there is at least one path, say $uCy$, in $uCy$, $yCx$ and $vP_{1}x$ whose length is at least $3$. Then there is a perfect matching of $G_{2}$ that contains $P_{2}$ and even edges of $uCy$. Here the even edges of $uCy$ are $2$-matching edges of $G_{1}$, so are $3^+$-matching edges of $G_{2}$. From Lemma \ref{L-sub-DMC}, $G$ also contains $3^+$-matching edges, a contradiction. So the length of $uCy$, $yCx$ and $vP_{1}x$ are all one. Meanwhile, $G_2$ is not regular. Thus $G_2\notin \mathfrak{B}$ by Lemma \ref{B-regular}, which implies $G\neq G_{2}$ and $P_{3}$ exists.

By Lemma \ref{L-two-ears} (iii), the ends of $P_{3}$, say $a$ and $b$, are inner vertices of $xCy$ and $P_{1}$ respectively. Without loss of generality, suppose that $a$ is in $xCy$ and $b$ is in $P_{1}$. If $a$ is black and $b$ is white ( $a$ may be $u$, and $b$ may be $v$), then the length of $bP_{1}y$ is at least $2$. Since there is a perfect matching of $G_{3}$ that contains $P_{3}$ and even edges of $bP_{1}y$, while those even edges are $2$-matching edges of $G_{2}$, so are $3^+$-matching edges of $G_{3}$ and $G$ by Lemma \ref{L-sub-DMC}, a contradiction.

If $a$ is white and $b$ is black, then the length of $bP_{1}x$ and $aCy$ are at least $2$. Suppose there is at least one path in $xCa$, $aCu$ and $bP_{1}y$ whose length is at least $3$, say, $xCa$ is that path. Then there is a perfect matching of $G_3$ that contains $P_{3}$ and the even edges of $xCa$. Here the even edges of $xCa$ are $2$-matching edges of $G_{1}$, so are the $3^+$-matching edges of $G_{3}$ and $G$, a contradiction. Hence the lengths of $xCa$, $aCu$ and $bP_{1}y$ are all one. If the length of $bP_{1}v$ is at least $3$, then there is a perfect matching of $G_{3}$ that contains $P_{3}$ and the even edges of $bP_{1}v$. Here the even edges of $bP_{1}v$ are $2$-matching edges of $G_{2}$, so are $3^+$-matching edges of $G_{3}$ and $G$, a contradiction. Therefore, $a$ and $v$ are adjacent in $P_{1}$. Therefore $G_{3}=K_{3,3}$.  Note that $K_{3,3}$ is matching double covered. Then $G$ maybe $K_{3,3}$.

If $G\neq K_{3,3}$, then $P_{4}$ exists, $G_{4}=K_{3,3}+P_{4}$. Let $V(K_{3,3})=\{x,y,u,v,a,b\}$. Since $K_{3,3}$ is complete, the ends of trivial ear $P_{4}$ are adjacent vertices in $K_{3,3}$, say, $u$ and $v$. Then the edge $ab$ is a $2$-matching edge of $K_{3,3}$, so it is a $3^+$-matching edge of $G_{4}$ and $G$ by Lemma \ref{L-sub-DMC}, a contradiction. Hence $P_{4}$ does not exist, $G=K_{3,3}$.

Combining the above cases, we can obtain that $G$ is $K_{2}^*$, $C_{4}^2$, $C_{4}^4$, $K_{3,3}$ or $C_{2n}$ if $G$ is bipartite in $\mathfrak{B}$. The proof is complete. \hfill $\Box$.

\section{Nonbipartite matching covered graphs}

\subsection{Bricks in $\mathfrak{B}$}
Let $H$ be a graph and let $x$ be a vertex of $H$. A graph $H':=H\{x\longrightarrow(b_1,b_2,\ldots,b_k)\}$ is said to be obtained from $H$ by {\it splitting $x$ into $b_1,b_2,\ldots,b_k$} if:

$(i)$ $V(H')=V(H-x)\cup\{b_1,b_2,\ldots,b_k\}$,

$(ii)$ $E(H')=E(H)$,

$(iii)$ $H'-\{b_1,b_2,\ldots,b_k\}=H-x$, and

$(v)$ every edge of $H$ that joins a vertex $v\in V(H-x)$ to $x$ in $H$ joins $v$ to one of the vertices $b_1, b_2,\ldots,b_k)$ in $H'$ (Thus ($\nabla_{H'}(b_1),\nabla_{H'}(b_2),\ldots,\nabla_{H'}(b_k))$ is a partition of $\nabla_{H}(x)$.)

Let us now introduce the four operations on bricks that can be used to generate all bricks from the three basic bricks (the complete graph $K_4$, the complement $\overline{C_6}$ of the 6-cycle and the Petersen graph $P$).

{\it Edge addition}: Let $H$ be a brick and let $x$ and $y$ be two distinct vertices of $H$. Obtain $G$ from $H$ by adding a new edge joining $x$ and $y$.

We shall refer to the following three operations as {\it vertex expansions}.

{\it Expansion of a vertex by a barrier of size two}: Let $H$ be a brick and let $x$ be a vertex of $H$ whose degree is at least four. Let $H':= H\{x\rightarrow (b_{1},b_{2})\}$ such that, in the underlying simple graph of $H'$, $b_{1}$ and $b_{2}$ have degree at least two. Now obtain $G$ from $H'$ by adding a new vertex $a$ and joining it to $b_{1}$ and $b_{2}$ and to a vertex $w$ of $H-x$ by an edge labelled $e$ (see Figure \ref{op2}). We shall refer to the graph $G$ thus constructed as a graph obtained from $H$ by {\it an expansion of $x$ by a barrier of size two}.

\vspace{2mm}\begin{figure}[h!]
  \centering
  \includegraphics[width=4 in]{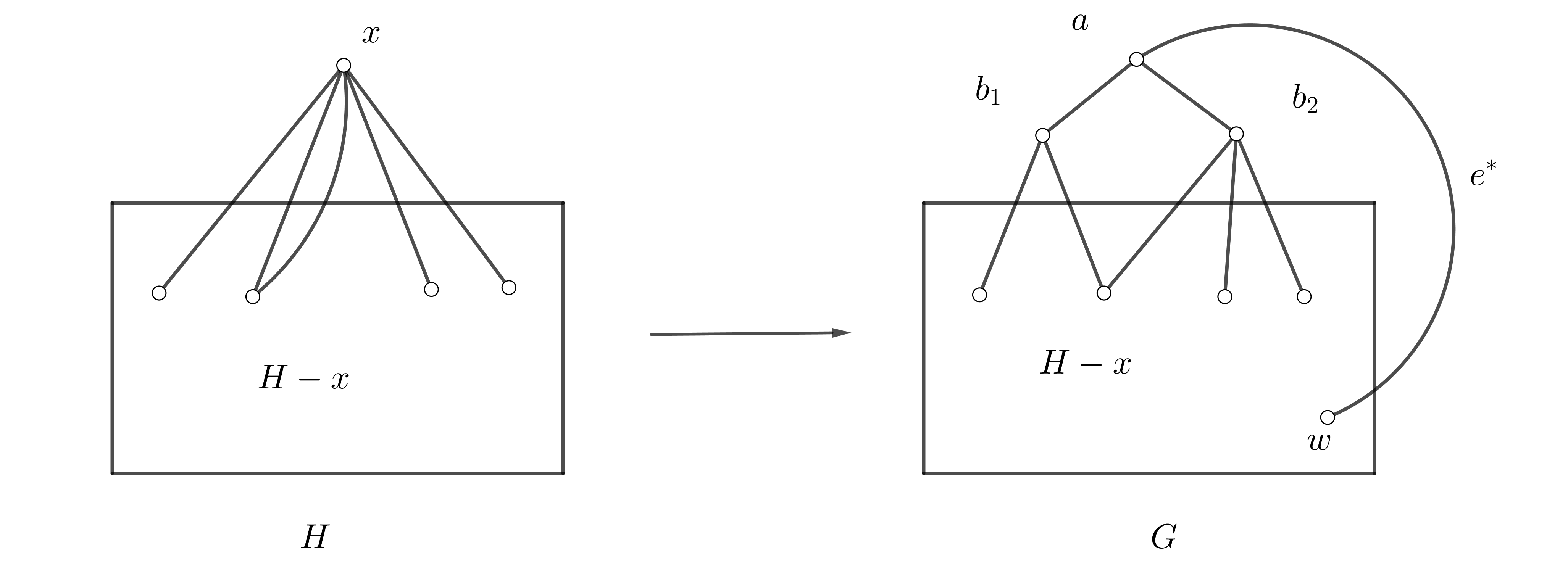}
   \vspace{-0.3cm}
  \caption{Expansion of a vertex by a barrier of size two}\label{op2}
\end{figure}

{\it Expansion of a vertex by a barrier of size three}:
Let $H$ be a brick and let $x$ be a vertex of $H$ whose degree is at least five. Let $H':=H\{x\rightarrow (b_{1},b_{2},b_{3})\}$ such that, in the underlying simple graph of $H'$, the degrees of $b_{1}$ and $b_{3}$ are at least two and the degree of $b_{2}$ is at least one. Now obtain $G$ from $H'$ by adding two new vertices $a_{1}$ and $a_{2}$, and joining $a_{1}$ to $b_{1}$ and $b_{2}$, joining $a_{2}$ to $b_{2}$ and $b_{3}$, and joining  $a_{1}$ and $a_{2}$ by an edge labelled $e$ (see Figure \ref{op3}). We shall refer to the graph $G$ thus constructed as a graph obtained from $H$ by {\it an expansion of $x$ by a barrier of size three}.

\begin{figure}[h!]
  \centering
  \includegraphics[width=4 in]{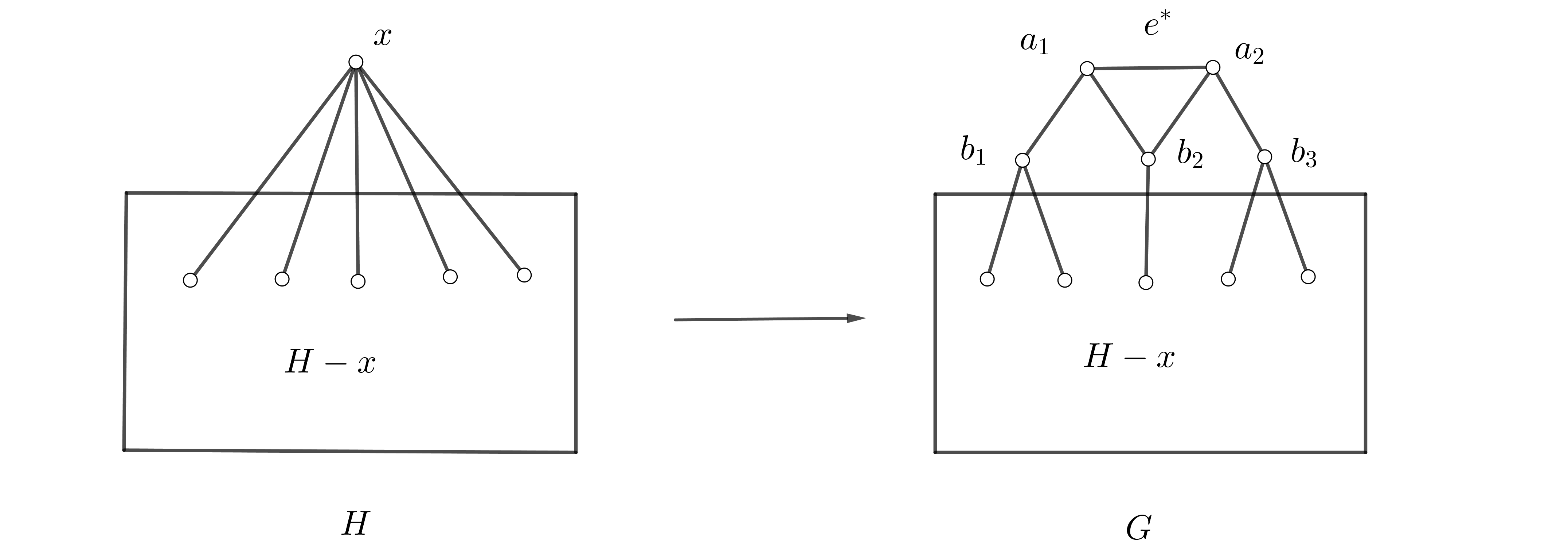}
    \vspace{-0.3cm}
  \caption{Expansion of $x$ by a barrier of size three}\label{op3}
\end{figure}

\begin{figure}[h!]
  \centering
  \includegraphics[width=4 in]{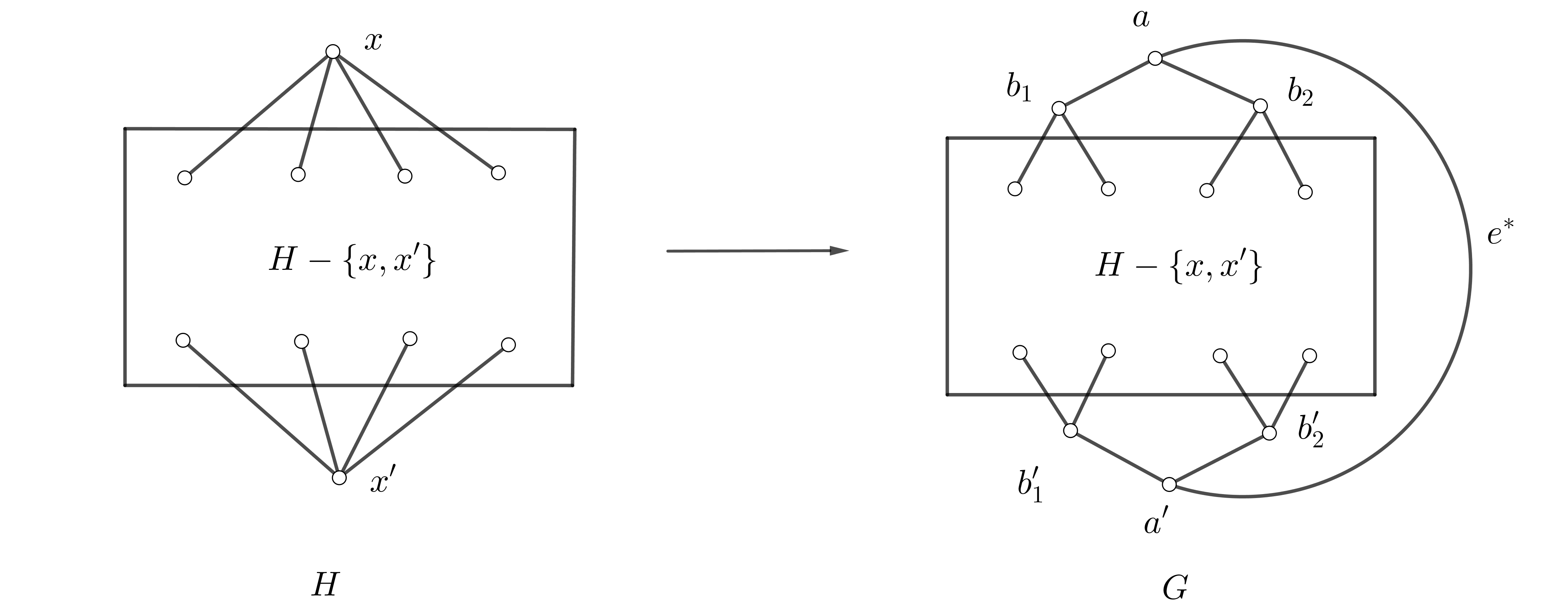}
    \vspace{-0.3cm}
  \caption{Expansion of $x$ and $x'$ by barriers of size two}\label{op4}
\end{figure}

{\it Expansion of two vertices by barriers of size two}:
Let $H$ be a brick and let $x$ and $x'$ be two vertices of $H$ whose degrees are at least four. Let $H':=H\{x\rightarrow (b_{1},b_{2}); x'\rightarrow (b_{1}',b_{2}')\}$ such that, in the underlying simple graph of $H'$, each of $b_{1}$, $b_{2}$, $b_{1}'$ and $b_{2}'$ has degree at least two and has at least one neighbour in $V(H-\{x,x'\})$.
Now obtain $G$ from $H'$ by adding two new vertices $a$ and $a'$, joining $a$ to $b_{1}$ and $b_{2}$, joining $a'$ to $b_{1}'$ and $b_{2}'$, and joining $a$ and $a'$ by an edge labelled $e$ (see Figure \ref{op4}). We shall refer to the graph $G$ thus constructed as a graph obtained from $H$ by {\it an expansion of $x$ and $x'$ by barriers of size two}. (Note that the order in which the two vertices $x$ and $x'$ are split is immaterial).

The study in this section is based on two very important results.

\begin{lemma}\cite{M2006}\label{brick H-G}
Let $H$ be a brick and let $G$ be a graph obtained from $H$ by means of edge additions and vertex expansions. Then $G$ is a brick.
\end{lemma}

\begin{lemma}\cite{M2006}\label{gena-brick}
Every brick different from the three basic bricks can be obtained from $K_{4}$, $\overline{C}_{6}$ or the Petersen graph by a sequence of applications of edge additions and vertex expansions.
\end{lemma}

Let $H$ be a brick. Suppose that edge $e$ is a $3^+$-matching edge of $H$, $M_{1}$, $M_{2}$ and  $M_{3}$ are three perfect matchings of $H$ containing $e$. If $G$ is obtained from $H$ by an edge addition, then $M_{i}$ $(i=1,2,3)$ are also three perfect matchings of $G$ containing $e$. If $G$ is obtained from $H$ by an expansion of a vertex by a barrier of size two (see Figure \ref{op2}), then $M_{i}\cup ab_{1}$ or $M_{i}\cup ab_{2}$ $(i=1,2,3)$ are perfect matchings of $G$ containing $e$. If $G$ is obtained from $H$ by an expansion of a vertex by a barrier of size three (see Figure \ref{op3}), then $M_{i}\cup a_{1}b_{2}\cup a_{2}b_{3}$ or $M_{i}\cup a_{1}b_{1}\cup a_{2}b_{3}$ or $M_{i}\cup a_{1}b_{1}\cup a_{2}b_{2}$ $(i=1,2,3)$ are perfect matchings of $G$ containing $e$. If $G$ is obtained from $H$ by an expansion of two vertices by barriers of size two (see Figure \ref{op4}), then $M_{i}\cup ab_{1}\cup a'b_{1}'$ or $M_{i}\cup ab_{1}\cup a'b_{2}'$ or $M_{i}\cup ab_{2}\cup a'b_{1}'$ or~$M_{i}\cup ab_{2}\cup a'b_{2}'$ $(i=1,2,3)$ are perfect matchings of $G$ containing $e$. In conclusion, $G$ has a $3^+$-matching edge. Through the above analysis and induction, we can get the following lemma.

\begin{lemma}\label{H-G}
Let $H$ be a brick, and $G$ be a graph obtained from $H$ by a sequence of applications of edge additions and vertex expansions. If there exists a $3^+$-matching edge in $H$,
then $G$ has a $3^+$-matching edge.
\end{lemma}

\begin{lemma}\label{operation2}
Let $H$ be a brick and $x\in V(H)$, and $G$ be a graph obtained from $H$ by an expansion of $x$ by a barrier of size two. If the degree of $x$ is four and there exists a $2$-matching edge incidents with $x$, or the degree of $x$ is at least five, then $G$ has a $3^+$-matching edge.
\end{lemma}

\begin{proof}
Let $d_H(x)$ denote the degree of $x$ in $H$. If $d_H(x)=4$, let $e_{1}$, $e_{2}$, $e_{3}$ and $e_{4}$ be the edges incident with $x$. Assume that $e_{1}$ and $e_{2}$ incident with $b_{1}$, $e_{3}$ and $e_{4}$ incident with $b_{2}$ in $G$. Suppose without loss of generality that $e_{1}$ is a $2$-matching edge in $H$, $M_{1}$ and $M_{2}$ are two perfect matchings of $H$ containing $e_{1}$. Then $M_{1}\cup ab_{2}$ and $M_{2}\cup ab_{2}$ are two perfect matchings of $G$ containing $ab_{2}$. Since $e_{2}$ is a admissible edge in $H$, let $M_ {3}$ be a perfect matching of $H$ containing $e_{2}$, thus $M_{3}\cup ab_{2}$ is a perfect matching of $G$ containing $ab_{2}$. Hence $ab_{2}$ is a $3^+$-matching edge of $G$.

If $d_H(x)\geq 5$, then at least one of $b_{1}$ and $b_{2}$ incidents with at least four edges in $G$, might as well let $b_{1}$ be that vertex and $e_{1}'$, $e_{2}'$, $e_{3}'$ and $ab_{2}$ be four edges incident with $b_1$. Note that $e_{1}'$, $e_{2}'$ and $e_{3}'$ are admissible edges in $H$, and are incident with $x$. Let $M_{i}$ be the perfect matchings in $H$ containing $e_{i}'$ $(i=1,2,3)$, then $M_{i}\cup ab_{2}$ $(i=1,2,3)$ are perfect matchings of $G$ containing $ab_{2}$. Therefore $ab_{2}$ is a $3^+$-matching edge of $G$.
\end{proof}

\begin{lemma}\label{operation3}
Let $H$ be a brick and $x\in V(H)$, and $G$ be a graph obtained from $H$ by an expansion of $x$ by a barrier of size three, then $G$ has a $3^+$-matching edge.
\end{lemma}

\begin{proof}
By the process of an expansion of $x$ by a barrier of size three, there exist three edges $e_{1}$, $e_{2}$ and $e_{3}$ incident with $x$ such that $e_1$ and $e_2$ are incident with $b_1$, $e_3$ is incident with $b_2$. Note that $e_{1}$, $e_{2}$ and $e_{3}$ are admissible edges in $H$. Let $M_{i}$ be the perfect matching of $H$ containing $e_{i}$ for $i=1,2,3$. Then $M_{1}\cup a_{1}b_{2}\cup a_{2}b_{3}$, $M_{2}\cup a_{1}b_{2}\cup a_{2}b_{3}$, and $M_{3}\cup a_{1}b_{1}\cup a_{2}b_{3}$ are three perfect matchings of $G$ containing $a_{2}b_{3}$. Therefore, $a_{2}b_{3}$ is a $3^+$-matching edge in $G$.
\end{proof}

\begin{lemma}\label{operation4}
Let $H$ be a brick and $x, x'\in V(H)$, and $G$ be the graph obtained from $H$ by an expansion of $x$ and $x'$ by barriers of size two. If one of the following conditions hold: (1) the degrees of $x$ and $x'$ are both $4$, and there exists a $2$-matching edge in $H$ incident with $x$ or $x'$; (2) the degree of $x$ or $x'$ is at least $5$, then $G$ has a $3^+$-matching edge.
\end{lemma}

%


\begin{proof}
If both $x$ and $x'$ have degree $4$, let $e_{1}$, $e_{2}$, $e_{3}$ and $e_{4}$ be the four edges incident with $x$, let $e_{1}'$, $e_{2}'$, $e_{3}'$ and $e_{4}'$ be the four edges incident with $x'$. Without loss of generality, assume that $e_{1}$ and $e_{2}$ are incident with $b_{1}$, $e_{3}$ and $e_{4}$ are incident with $b_{2}$, and $e_{1}'$ and $e_{2}'$ are incident with $b_{1}'$, $e_{3}'$ and $e_{4}'$ are incident with $b_{2}'$. If there exists a $2$-matching edge in $H$ incident with $x$ or $x'$, say, $e_{1}$. Let $M_{1}$ and $M_{2}$ be two perfect matchings of $H$ containing $e_{1}$. Then $M_{1} \cup ab_{2} \cup a'b_{i}'$ and $M_{2} \cup ab_{2} \cup a'b_{i}'$ ($i=1$ or $2$) are perfect matchings of $G$ containing $ab_{2}$. Since $e_{2}$ is a admissible edge in $H$, there exists a perfect matching $M_{3}$ of $H$ containing $e_{2}$. Then $M_{3} \cup ab_{2} \cup a'b_{i}'$ ($i=1$ or $2$) is a perfect matching of $G$ containing $ab_{2}$. Therefore, $ab_{2}$ is a $3^+$-matching edge in $G$.

If the degree of $x$ or $x'$ is at least $5$, say $x$, then at least one of $b_{1}$ and $b_{2}$ is incident with at least four edges in $G$, say $b_1$. Let $e_{1}$, $e_{2}$ and $e_{3}$ be three edges incident with $x$ in $H$, and incident with $b_1$ in $G$. Since $e_{1}$, $e_{2}$ and $e_{3}$ are admissible edges in $H$, there exist perfect matchings $M_{1}$, $M_{2}$ and $M_{3}$ in $H$ containing $e_{1}$, $e_{2}$ and $e_{3}$ respectively. Then $M_{i} \cup ab_{2} \cup a'b_{1}'$ or $M_{i} \cup ab_{2} \cup a'b_{2}'$ for $i=1,2,3$ are three perfect matchings of $G$ containing $ab_{2}$. Therefore, $ab_{2}$ is a $3^+$-matching edge in $G$.
\end{proof}

Notice that an edge addition for a matching double covered brick can generate $3^+$-matching edges. Combining this fact with Lemmas \ref{operation2}, \ref{operation3} and \ref{operation4}, we can obtain the result as follow.

\begin{lemma}\label{D-H-G}
Let $H$ be a matching double covered brick. If $G$ is obtained from $H$ by series of edge additions and vertex expansions, then $G$ contains a $3^+$-matching edge.
\end{lemma}

\begin{lemma}\label{K-G}
Let $K_{4}^*$ be the graph obtained from $K_4$ by edge additions. If there exist at least three parallel edges incident with two adjacent vertices in $K_{4}^*$, then $K_{4}^*$ has a $3^+$-matching edge.
\end{lemma}

\begin{proof}
Let $v_{1}$, $v_{2}$, $v_{3}$ and $v_{4}$ be the four vertices of $K_{4}^*$. Suppose that there exist at least three parallel edges incident with two adjacent vertices, say $v_{1}$ and $v_{2}$. Then edge $v_3v_4$ is a $3^+$-matching edge of $K_{4}^*$.
\end{proof}

{\bf Proof of Theorem \ref{brick-B}.} By definition, one can check that $K_{4}$, $K_{4}^2$, $K_{4}^4$, the Petersen graph $P$, $T_{1}$ and $K_{4}^6$ are in $\mathfrak{B}$. Hence the sufficiency follows. Before we get to the necessity, let us provide a claim will be used in the subsequent argument. The claim can be obtained by Lemmas \ref{operation2}, \ref{operation3} and \ref{operation4}.

{\bf Claim.} Let $G$ be a brick containing no $3^+$-matching edges. In the generation process of $G$ obtained from a brick by means of edge additions and vertex expansions, if there exists a graph $F$ such that each vertex with degree four is incident with a $2$-matching edge, then the new graph followed by $F$ can only be obtained from $F$ by edge additions.

Now we proceed to establish the necessity. Suppose that $G$ is a brick in $\mathfrak{B}$. Note that $K_{4}$ and the Petersen graph are all in $\mathfrak{B}$, $\overline{C}_{6}$ is not in $\mathfrak{B}$. Then $G$ may be $K_{4}$ or the Petersen graph. Otherwise, by Lemma \ref{gena-brick}, the graph $G$ can be obtained from $K_{4}$, $\overline{C}_{6}$ or the Petersen graph by a sequence of applications of edge additions and vertex expansions. We consider the following cases based on the kinds of the basic bricks.

{\bf Case 1.} If $G$ is obtained from the Petersen graph $P$ by a series of edge additions and vertex expansions. Note that $P$ is matching double covered. Then by Lemma \ref{D-H-G}, the first graph $H_1$, during the generation process of $G$ obtained from $P$ by means of edge additions and vertex expansions, can only be obtained from $P$ by an edge addition. However, we can check that $H_1$ has a $3^+$-matching edge. Hence by Lemma \ref{H-G}, $G$ has a $3^+$-matching edge and so is not in $\mathfrak{B}$, a contradiction.

{\bf Case 2.} If $G$ is obtained from $\overline{C}_6$ by a series of edge additions and vertex expansions. Since $\overline{C}_6$ is $3$-regular and $\overline{C}_6\notin \mathfrak{B}$, the first graph $H_1$ can only be obtained from $\overline{C}_6$ by adding edge $e_1$, i.e. $H_1=\overline{C}_6+e_1$. Let $V(\overline{C}_6)=\{v_1,v_2,\ldots,v_6\}$, see Figure \ref{T13} (1). Suppose that the ends of $e_1$ are two adjacent vertices in $\overline{C}_6$, say $v_1$ and $v_2$, or $v_1$ and $v_4$. If $e_1$ connects $v_1$ and $v_2$, then $v_5v_6$ is a $3^+$-matching edge of $H_1$. If $e_1$ connects $v_1$ and $v_4$, then $v_1v_6$ is a $3^+$-matching edge of $H_1$. Then by Lemma \ref{H-G}, the graph $G$ obtained from $H_1$ by a series of edge additions and vertex expansions also has a $3^+$-matching edge, a contradiction. Therefore, the ends of $e_1$ are two nonadjacent vertices in $\overline{C}_6$. By symmetry, we can suppose that the ends of $e_1$ are $v_1$ and $v_3$ ( see Figure \ref{T13} (2)).

\begin{figure}[h!]
  \centering
  \includegraphics[width=4in]{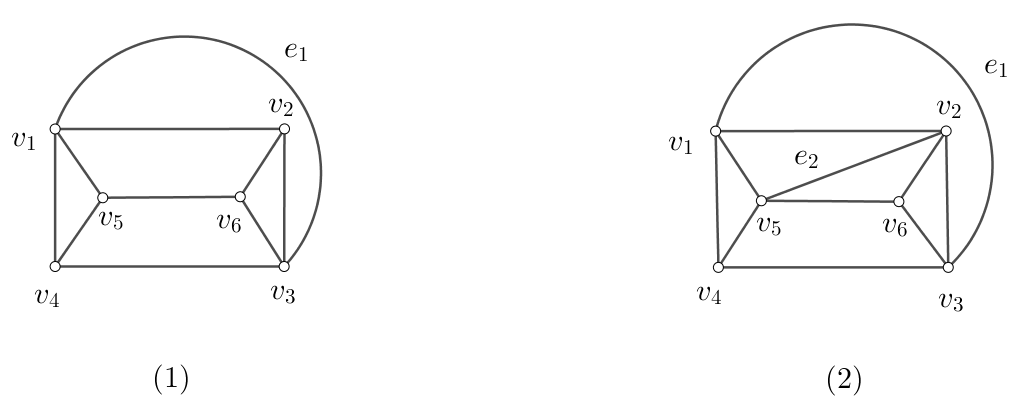}
    \vspace{-0.5cm}
  \caption{(1)The graph $\overline{C}_{6}$; (2)The graph $H_1=\overline{C}_{6}+e_1$; (3) the brick $H_2=\overline{C}_{6}+e_1+e_2$.}\label{T13}
\end{figure}

Since $H_1$ is not regular and $G\in \mathfrak{B}$, $H_1$ can not be $G$ by Lemma \ref{B-regular}. Note that $H_{1}$ only has two $4$-degree vertices $v_{1}$ and $v_3$. Meanwhile, $v_{1}v_{2}$ and $v_{3}v_{4}$ are $2$-matching edges incident with $v_{1}$ and $v_3$ in $H$ respectively. Thus by Claim, the new graph $H_2$ followed by $H_1$ can only be obtained by an edge addition. Let $H_{2}=H_{1}+e_{2}$. Based on the above analysis, $e_{2}$ cannot connect two adjacent vertices in $H_{1}$. Therefore, $e_{2}$ is $v_{1}v_{6}$, $v_{2}v_{4}$ or $v_{2}v_{5}$ by symmetry. If $e_{2}=v_{1}v_{6}$, then $v_{4}v_{5}$ is a $3^+$-matching edge in $H_{2}$. Similarly, $e_{2}$ cannot be $v_{2}v_{4}$. If $e_{2}=v_{2}v_{5}$, as shown in Figure \ref{T13} (2). Since $H_2$ is not regular, we have $G \neq H_2$ by Lemma \ref{B-regular}. Note that $H_2$ has four $4$-degree vertices $v_1$, $v_2$, $v_3$ and $v_5$, and $v_1v_2$, $v_4v_5$ and $v_3v_4$ are $2$-matching edges in $H_2$. Therefore by Claim, the new graph $H_3$ followed by $H_2$ can only be obtained by adding an edge $e_3$. By the same analysis as above, $e_3$ connects only $v_4$ and $v_6$, otherwise there exist $3^+$-matching edges in $H_3$. Thus $H_3=T_1$. One can check that $T_1$ is matching double covered and so is in $\mathfrak{B}$. Hence $G$ may be $T_1$. If $G\neq T_1$. then $G$ has a $3^+$-matching edge by Lemma \ref{D-H-G}, a contradiction. Therefore, $G=T_1$ in this case.

{\bf Case 3.} If $G$ is obtained from $K_4$ by a series of edge additions and vertex expansions. Let $V(K_4)=\{v_1, v_2, v_3, v_4\}$. Since $K_4$ is $3$-regular, the first graph $H_1$ is obtained from $K_4$ by adding an edge $e_{1}$, say $e_1=v_1v_4$. Let the other edge with ends $v_1$ and $v_4$ be $e'$ ( see Figure \ref{T10}). Obviously $G\neq H_1$ since $H_1$ is not regular. In $H_1$, the degree of $v_1$ and $v_4$ in $H_1$ are four and no $2$-matching edges incident with $v_1$ or $v_4$, the degree of $v_2$ and $v_3$ are three. Hence by Lemmas \ref{operation2}, \ref{operation3}, and \ref{operation4}, we can obtain that the new graph $H_2$ is obtained from $H_1$ by an edge addition, an expansion of a vertex by a barrier of size two, or an expansion of two vertices by barriers of size two. We consider these three cases in the following.

\begin{figure}[h!]
  \centering
  \includegraphics[width=1.5 in]{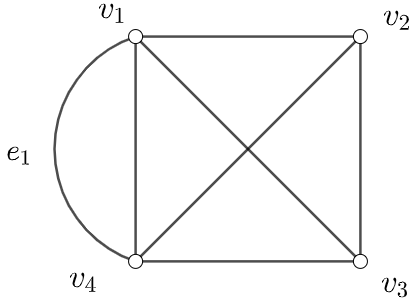}
   \vspace{-0.3cm}
  \caption{The brick $H_{1}$.}\label{T10}
\end{figure}

{\bf Case 3.1.} If $H_2$ is obtained from $H_1$ by an edge addition. Let $H_2=H_{1}+e_{2}$. From Lemma \ref{K-G}, the ends of $e_{2}$ can be $v_{1}$ and $v_{2}$ or $v_{2}$ and $v_{3}$ by symmetry.

{\bf Subcase 3.1.1.} If the ends of $e_{2}$ are $v_{1}$ and $v_{2}$. Since $H_2$ is not regular, $H_2$ is not in $\mathfrak{B}$ and $G\neq H_2$. Note that $H_2$ has a $5$-degree vertex $v_{1}$, two $4$-degree vertices $v_{2}$ and $v_{4}$ both incident with a $2$-matching edge. By Claim, the new graph $H_3$ can only be obtained from $H_2$ by an edge addition. Set $H_{3}=H_2+e_{3}$. From Lemma \ref{K-G}, the ends of $e_{3}$ can be $v_{1}$ and $v_{3}$, $v_{2}$ and $v_{3}$, or $v_{2}$ and $v_{4}$ by symmetry. No matter which case occurs, one can check that $H_3$ is not regular and $G\neq H_3$. Moreover, each $4$-degree vertex in $H_3$ is incident with a $2$-matching edge. Hence the new graph $H_4$ can only be obtained from $H_3$ by an edge addition by Claim. Set $H_{4}=H_{3}+e_{4}$. Repeat the above process, we can obtain the graph $K_{4}^6$ by edge additions only. Note that $K_{4}^6$ is a matching double covered brick. Hence $G$ may be $K_{4}^6$. If $G\neq K_{4}^6$, then $G$ must contain $3^+$-matching edges by Lemma \ref{D-H-G}, a contradiction. Hence $G=K_{4}^6$.

{\bf Subcase 3.1.2.} If the ends of $e_{2}$ are $v_{2}$ and $v_{3}$. By the same arguement as that in Subcase 3.1.1, we can obtain that $G$ may be $K_4^2$ or $K_4^4$.

{\bf Case 3.2.} If $H_2$ is obtained from $H_1$ by an expansion of a vertex by a barrier of size two. Without loss of generality, suppose that $v_1$ is split into $b_{1}$ and $b_{2}$, the new vertex $a$ is incident with $b_1$ and $b_2$, the new edge $e_{2}$ connects $a$ and a vertex in $H_{1}-v_{1}$. Base on the distribution of $e_{2}$ and the parallel edges between $v_{1}$ and $v_{4}$, the graph $H_2$ is isomorphic to one of graphs shown in Figure \ref{T11}.

\begin{figure}[h!]
  \centering
  \includegraphics[width=4.5 in]{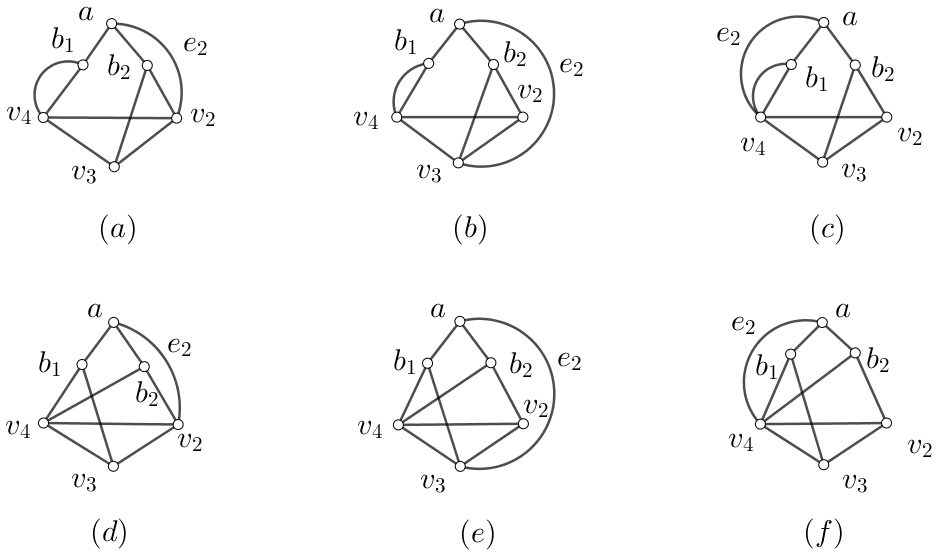}
   \vspace{-0.3cm}
  \caption{The graphs obtained from $H_{1}$ by an expansion of a vertex of a barrier of size two.}\label{T11}
\end{figure}

If $H_2$ is isomorphic to (a) or (b), the edge $b_2v_3$ or $b_2v_2$ is a $3^+$-matching edge in $H_2$. Hence by Lemma \ref{H-G}, the graph $G$ contains a $3^+$-matching edge, a contradiction. If $H_2$ is isomorphic to (c), then $e_2$ is not an admissible edge contradicting the fact that $H_2$ is a brick.

If $H_2$ is isomorphic to (d), then $H_2$ is isomorphic to the graph shown in Figure \ref{T13} (2). Through the same argument as that in Case 2, we obtain that $G=T_1$. If $H_2$ is isomorphic to (e), we can also prove that $G=T_1$.

If $H_2$ is isomorphic to (f), then $H_2$ is not regular and $G\neq H_2$. Since $H_2$ has only a $5$-degree vertex $v_{4}$, and the other five vertices are all $3$-degree vertices, the new graph $H_3$ followed by $H_2$ can only be obtained by an edge addition by Claim and Lemma \ref{operation3}. Let $H_3=H_2+e_{3}$. If the ends of $e_3$ are two adjacent vertices in $H_2$, say, $a$ and $b_1$, then $ab_1v_4b_2$ is a $4$-cycle containing parallel edges with ends $a$ and $b_1$. One can check that $v_2v_3$ is a $3^+$-matching edge of $H_3$. Hence by Lemma \ref{H-G},  $G$ contain $3^+$-matching edges, a contradiction.  Similarly, if the ends of $e_3$ are two other adjacent vertices in $H_2$, we can also get a contradiction. Therefore, the ends of $e_3$ are two nonadjacent vertices in $H_2$. If $e_{3}=av_{2}$, then $b_{1}v_{3}$ is a $3^+$-matching edge in $H_3$. If $e_{3}=av_3$, then $b_{2}v_{2}$ is a $3^+$-matching edge in $H_3$. If $e_{3}=b_1b_2$, then $v_{2}v_{3}$ is a $3^+$-matching edge in $H_3$. If $e_{3}=b_1v_2$, then  $ab_{2}$ is a $3^+$-matching edge in $H_3$. If $e_{3}=b_2v_3$, then $ab_{1}$ is a $3^+$-matching edge in $H_3$. In the above cases, we can always find a $3^+$-matching edge in $H_3$. Thus by Lemma \ref{H-G}, $G$ contain $3^+$-matching edges, a contradiction.

{\bf Case 3.3.} If $H_2$ is obtained from $H_1$ by an expansion of $v_1$ and $v_4$ by barriers of size two. Suppose that $v_{1}$ is split into $b_{1}$ and $b_{2}$, and $v_{4}$ is split into $b_{1}'$ and $b_{2}'$ in $H_2$. The new vertex $a$ is incident with $b_1$ and $b_2$, and the another new vertex $a'$ is incident with $b_{1}'$ and $b_{2}'$. Denote by $e_{2}$ the new edge connecting $a$ and $a'$ in $H_2$. Since $b_{1}$, $b_{2}$, $b'_{1}$ and $b'_{2}$ have at least one neighbor in $V(H_{1}-\{v_{1},v_{4}\})$, the edges with ends $v_{1}$ and $v_{4}$ cannot all be assigned to $b_{1}$ or $b_{2}$, nor can they all be assigned to $b'_{1}$ or $b'_{2}$. Therefore, $H_2$ is isomorphic to one of the following graphs shown in Figure \ref{T12}. We can check that $v_2v_3$ is a $3^+$-matching edge of $H_2$.  Hence by Lemma \ref{H-G}, $G$ contains a $3^+$-matching edge, a contradiction.

\begin{figure}[h!]
  \centering
  \includegraphics[width=4 in]{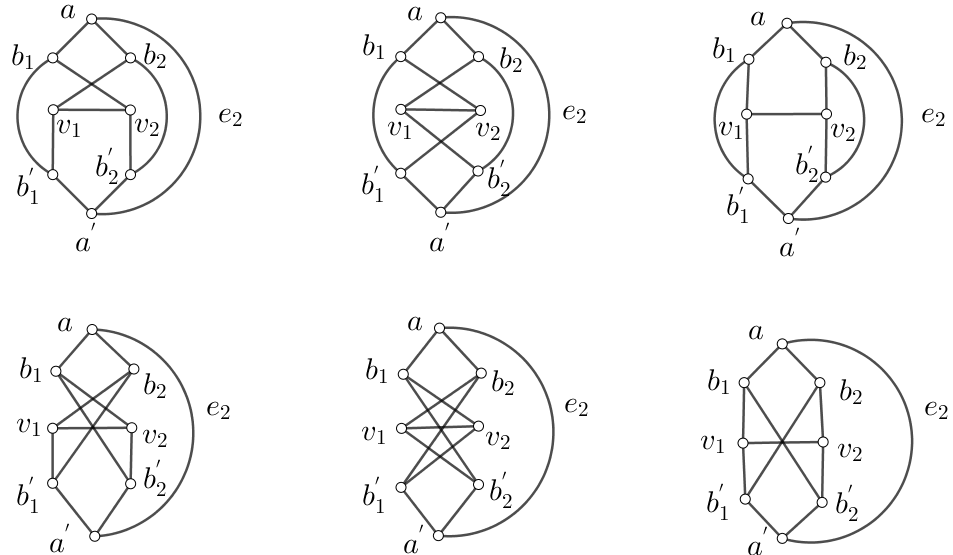}
  \vspace{-0.4cm}
  \caption{The $H_2$ obtained from $H_{1}$ by an expansion of two vertices by barriers of size two.}\label{T12}
\end{figure}

Combining the above cases, we can obtain that $G$ is $K_{4}$, $K_{4}^2$, $K_{4}^4$, the Petersen graph $P$, $T_{1}$ or $K_{4}^6$ if $G$ is a brick in $\mathfrak{B}$. The proof is complete. \hfill $\Box$.

\subsection{Matching covered graphs with nontrivial tight cuts in $\mathfrak{B}$}

We have characterized bricks in $\mathfrak{B}$ in Section 3.1. Now we consider the matching covered graph $G$ with nontrivial tight cuts. Before giving our results, we first present some significant lemmas will be used in the sequel.

\begin{lemma}\label{tightcut-i-edge}
Let $G$ be a matching covered graph with a nontrivial tight cut $C$. Denote by $G_{1}$ and $G_{2}$ the two $C$-contractions of $G$.  \\
(1) For any edge $f$ in $C$, if $f$ is a $i$-matching edge in $G_{1}$, and is a $j$-matching edge in $G_{2}$, then $f$ is a $ij$-matching edge in $G$;\\
(2) If $G_1$ or $G_2$ contain $3^+$-matching edges, then $G$ also contains $3^+$-matching edges.\\
(3) If $G$ is a graph in $\mathfrak{B}$ and $f$ is a $i$-matching edge in $G_{j}$, then every edge in the perfect matching of $G_{j}$ containing $f$ is also a $i$-matching edge in $G_{j}$, where $i,j \in\{1,2\}$.
\end{lemma}

\begin{proof}
For (1), let $M_1, M_2, \ldots, M_i$  be  perfect matchings of $G_1$ containing $f$, and $N_1, N_2, \ldots, N_j$ be perfect matchings of $G_2$ containing $f$. Then $M_k\cup N_l$ is perfect matching of $G$ containing $f$, where $1\leq k\leq i$, $1\leq l\leq j$. Hence $f$ is a $ij$-matching edges of $G$.

For (2), suppose that there exists a $3^+$-matching edge $e$ in $G_{1}$ or $G_{2}$, say $e\in E(G_1)$. Let $M_{1}$, $M_{2}$ and $M_{3}$ be three perfect matchings in $G_{1}$ containing $e$, and $e_{i}=M_{i}\cap C$, where $i=1,2,3$ ($e_{1}, e_{2}, e_{3}$ may be the same edge). Since $C$ is a nontrivial tight cut of $G$, $G_{1}$ and $G_{2}$ are matching covered. Hence $e_{i}$ is a admissible edge in $G_{2}$ for $1\leq i\leq 3$. Let $M'_{i}$ be the perfect matching in $G_{2}$ containing $e_{i}$, then $M_{i}\cup M'_{i}$ be the perfect matchings in $G$ containing $e$ for $1\leq i\leq 3$. Then $e$ is also a $3^+$-matching edge of $G$.

For (3), without loss of generality, we may suppose that there exists a perfect matching $M_1$ in $G_1$ such that $M_1$ contains a $1$-matching edge $e$ and a $2$-matching edge $f$. The other perfect matching of $G_1$ containing $f$ is denoted by $M_2$. Let $M_1\cap C=\{a\}$ and $M_2\cap C=\{b\}$ (may $a=b$). Since $C$ is a tight cut, $b$ is a admissible edge in $G_2$. Let $N'$ be the perfect matching in $G_2$ that contains $b$. If $a$ is a $1$-matching edge in $G_2$, and $N_1$ is the unique perfect matching of $G_2$ that contains $a$ (when $a=b$, $N_1=N'$). Then $M_1\cup N_1$ is the unique perfect matching in $G$ that contains $e$ and $f$, and $M_2\cup N'$ is the perfect matching in $G$ that contains $f$. Hence $e$ is a $1$-matching edge and $f$ is a $2$-matching edge in $M_1\cup N_1$ of $G$ contradicting the fact that $G$ is in $\mathfrak{B}$. If $a$ is contained in two perfect matchings $N_{1}$ and $N_{2}$ in $G_2$, then $M_{1}\cup N_1$, $M_1\cup N_2$, and $M_2\cup N'$ (may $N'=N_1$ or $N'=N_2$) are three perfect matchings in $G$ that contain $f$, a contradiction again. Thus (3) holds.
\end{proof}

\begin{lemma}\label{tight-cut&B}
Let $G$ be a graph with nontrivial tight cuts in $\mathfrak{B}$, then every graph produced in the tight cuts decomposition process belongs to $\mathfrak{B}$.
\end{lemma}

\begin{proof}
Let $C$ be a nontrivial tight cut of $G$, $G_{1}$ and $G_{2}$ be two $C$-contractions of $G$. Note that $G\in \mathfrak{B}$, then $G$ is matching covered, so are $G_1$ and $G_2$. By Lemma \ref{tightcut-i-edge} (2), $G_{1}$ and $G_{2}$ contain no $3^+$-matching edges, otherwise $G$ contains $3^+$-matching edges, a contradiction. Moreover, if $f$ is a $i$-matching edge in $G_{j}$, then by Lemma \ref{tightcut-i-edge}, every edge in the perfect matching of $G_{j}$ containing $f$ is also a $i$-matching edge in $G_{j}$, where $i,j \in\{1,2\}$. Hence $G_{1}$ and $G_{2}$ are graphs in $\mathfrak{B}$. By the above analysis, we can obtain that every graph produced in the tight cuts decomposition process belongs to $\mathfrak{B}$.
\end{proof}

\begin{lemma}\label{tight-cut}
Let $G$ be a matching double covered graph with a nontrivial tight cut $C$, $G_{1}$ and $G_{2}$ be two $C$-contractions of $G$. Then $G_{1}$, $G_{2}\in\mathfrak{B}$. For any edge $f$ in $C$, if $f$ is a $i$-matching edge in $G_{1}$, then $f$ is a $(3-i)$-matching edge in $G_{2}$, where $i\in \{1,2\}$.
\end{lemma}

\begin{proof}
Obviously $G$ is in $\mathfrak{B}$, then $G_{1}$, $G_{2}\in\mathfrak{B}$ by Lemma \ref{tight-cut&B}. Suppose that $f$ is a $i$-matching edge in both $G_1$ and $G_2$ for $1\leq i\leq 2$, then by Lemma \ref{tightcut-i-edge} (1), $f$ is a $1$-matching edge or a $4$-matching edge of $G$, a contradiction.
\end{proof}

\begin{lemma}\label{C-DMC}
Let $G$ be a matching covered graph with nontrivial tight cut $C$, $G_{1}$ and $G_{2}$ be two $C$-contractions of $G$. \\
(1) If the numbers of 2-matching edges incident with the contraction vertices in $G_{1}$ and $G_{2}$ are more than half, then $G\notin \mathfrak{B}$.\\
(2) If $G\in \mathfrak{B}$ and $G_1$ is matching double covered, then $G_2$ is $K_{4}$ or an even cycle.
\end{lemma}

\begin{proof}
Suppose that $x$ and $\overline{x}$ are the contraction vertices with degree $d$ of $G_{1}$ and $G_{2}$ respectively. Then $G$ can be obtained by splicing $G_{1}$ and $G_{2}$ at $x$ and $\overline{x}$. For (1), if the number of $2$-matching edges incident with $x$ and $\overline{x}$ of $G_{1}$ and $G_{2}$ is more than $d/2$, then there exist a $2$-matching edge $ux$ of $G_{1}$ and a $2$-matching edge $\overline{x}v$ of $G_{2}$ such that $uv$ is an edge of $G$. Let $M_{1}$ and $M_{2}$ be two perfect matchings of $G_{1}$ containing $ux$, and $M_{1}'$ and $M_{2}'$ be two perfect matchings of $G_{2}$ containing $\overline{x}v$. Hence $(M_{1}-ux)\cup (M_{1}'-\overline{x}v)\cup uv$, $(M_{1}-ux)\cup (M_{2}'-\overline{x}v)\cup uv$, $(M_{2}-ux)\cup (M_{1}'-\overline{x}v)\cup uv$ and $(M_{2}-ux)\cup (M_{2}'-\overline{x}v)\cup uv$ are four perfect matchings in $G$ containing $uv$. Hence $uv$ is a $3^+$-matching edge of $G$, so $G\notin\mathfrak{B}$.

For (2), since $G\in \mathfrak{B}$, we have $G_2\in \mathfrak{B}$ by Lemma \ref{tight-cut&B}. Note that the tight cut $C$ is nontrivial, then $G_2$ is not $K_{2}^*$. If $G_2$ is not $K_{4}$ or an even cycle, then $G_2$ there exist $2^+$-matching edges by Lemma \ref{uniqe}. Note that $G_1\in \mathfrak{B}$ and every vertex in $G_1$ is incident with a $2$-matching edge. Hence there exist a $2$-matching edge $ux$ of $G_{1}$ and a $2$-matching edge $\overline{x}v$ of $G_{2}$ such that $uv$ is an edge of $G$. Through the same argument as that in (1), we can obtain that $uv$ is a $3^+$-matching edge of $G$, so $G\notin\mathfrak{B}$, a contradiction. Hence we can obtain that $G_2$ is $K_{4}$ or an even cycle.
\end{proof}

Let $G$ be a nonbipartite matching covered graph with nontrivial tight cuts in $\mathfrak{B}$. Then by Lemma \ref{tight-cut&B}, one can achieve that the bricks and braces obtained from $G$ through the tight cut decomposition procedure are also in $\mathfrak{B}$. Combining Lemmas \ref{B-regular}, \ref{C-DMC} and Theorems \ref{brace-B}, \ref{brick-B}, we can get that $G$ is $k$-regular graph in $\mathfrak{B}$, where $k\in\{2,3,4\}$. If $G$ is $2$-regular in $\mathfrak{B}$, one can obtain that $G$ is an even cycle since $G$ is matching covered. However, it contradicts with the fact that $G$ is nonbipartite. Hence we only need to consider the cases that $G$ is $k$-regular for $k=3$ and $k=4$ and achieve the following results, which proofs level in the next subsection.

\begin{theorem}\label{3-regular}
Let $G$ be a nonbipartite $3$-regular matching covered graph with a nontrivial cut. Then $G$ is in $\mathfrak{B}$ if and only if $G$ is $C_{\bigtriangleup}^1$, $T_2$, $T_4$, $T_5$ or $T_6$ (see Figure \ref{$T_{5}-T_{8}$}).
\end{theorem}

\begin{theorem}\label{4-regular}
Let $G$ be a nonbipartite $4$-regular matching covered graph with a nontrivial cut. Then $G$ is in $\mathfrak{B}$ if and only if $G$ is $T_3$ (see Figure \ref{$T_{5}-T_{8}$}).
\end{theorem}

Based on the results of Theorems \ref{3-regular} and \ref{4-regular}, The main result of this section (Theorem \ref{TH3}) follows directly.

\subsubsection{Proofs of Theorems \ref{3-regular} and \ref{4-regular}}

Let $G$ and $H$ be two vertex-disjoint graphs. For $u\in V(G)$ and $v\in V(H)$ with same degree $d$, let $u_{1}u, u_{2}u,\ldots,u_{d}u$ be the edges incident with $u$ in $G$ and let $v_{1}v, v_{2}v, \ldots,v_{d}v$ be the edges incident with $v$ in $H$.
Add new edges joining $u_{i}$ in $G-u$ and $v_{i}$ in $H-v$ for $1\leq i \leq d$.
The resulting graph is called {\it the splicing of $G$ and $H$ at $u$ and $v$}, denoted by $G(u)\odot H(v)$. In particular, if $H=K_{4}$ and $u$ is a $3$-degree vertex in $G$, the splicing of $G$ and at $u$ with $K_{4}$, denoted by $G(u)\odot K_{4}$, is called {\it the triangle replacement on $G$ at $u$}. The replaced triangle at $u$ is called the {\it replacement triangle}. Denote by $G\odot K_{4}$ the splicing of a vertex-transitive graph $G$ and $K_{4}$. If $G$ and $H$ are matching covered graphs, then $G(u)\odot H(v)$ is also matching covered.

Let $C_{\bigtriangleup}^1=C_4^2\odot K_{4}$, where $C_4^2$ is depicted in Figure \ref{T6}. Let $C_{\bigtriangleup}^2$ be the graph obtained from $C_4^2$ by performing a triangle replacement on two nonadjacent vertices or two adjacent vertices that not the end vertices of parallel edges. It is easy to verify that $C_{\bigtriangleup}^1$ and $C_{\bigtriangleup}^2$ belong to $\mathfrak{B}$.

\begin{lemma}\label{no-c-tri}
If $G'$ is a graph obtained from $C_{\bigtriangleup}^1$ by performing a triangle replacement at a vertex that is not the end of the multiple edges, then $G'$ contains a $3^+$-matching edge.
\end{lemma}

\begin{figure}[h!]
  \centering
  \includegraphics[width=5 in]{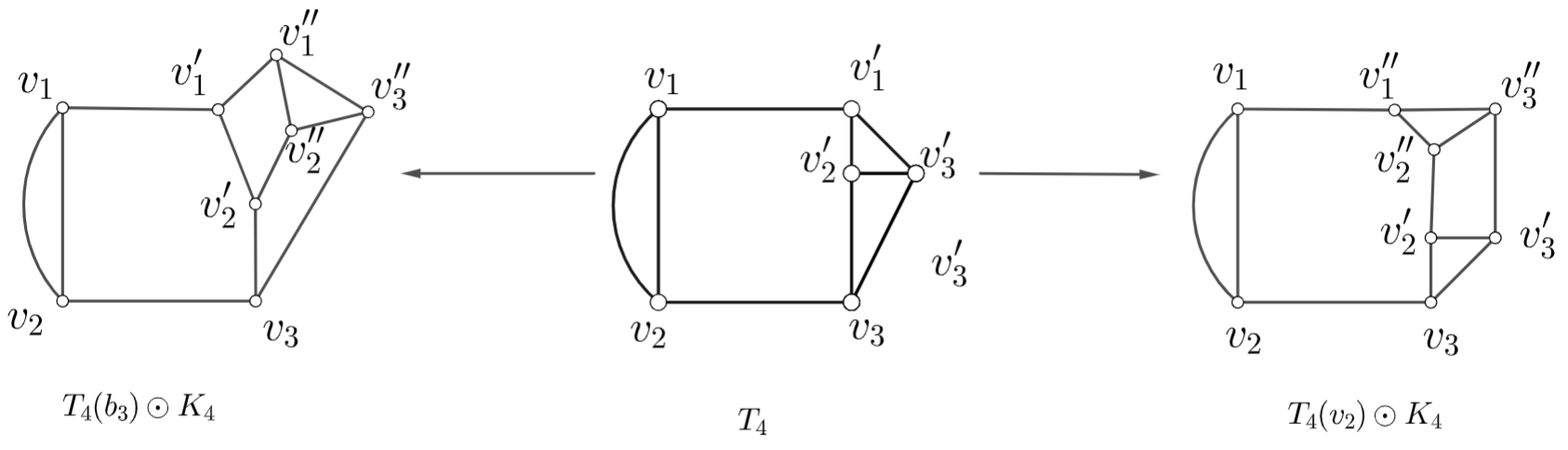}
   \vspace{-0.5cm}
  \caption{Graphs obtained from $C_{\bigtriangleup}^1$ by performing a triangle replacement}\label{T15}
\end{figure}

\begin{proof}
Let $e$ and $f$ be the multiple edges with ends $v_1$ and $v_2$ in $C_{\bigtriangleup}^1$. Let $v_3$, $v_1'$, $v_2'$ and $v_3'$ be the other four vertices, see Figure \ref{T15}. Let $v_1''$, $v_2''$ and $v_3''$ be the three vertices of the triangle replacement in $G'$. By symmetric, $G'$ is isomorphic to $C_{\bigtriangleup}^1(v_1')\odot K_4$ or $C_{\bigtriangleup}^1(v_3')\odot K_4$, see Figure \ref{T15}. If $G'\cong C_{\bigtriangleup}^1(v_1')\odot K_4$, then $\{e, v_2'v_2'', v_3v_3', v_1''v_3''\}$, $\{f, v_2'v_2'', v_3v_3', v_1''v_3''\}$ and $\{v_1v_1'', v_2'v_2'', v_2v_3, v_3'v_3''\}$ are three perfect matchings in $G'$ that contain $v_2'v_2''$. If $G'\cong C_{\bigtriangleup}^1(v_3')\odot K_4$, then $\{e, v_2'v_3, v_1'v_1'', v_2''v_3''\}$, $\{f, v_2'v_3, v_1'v_1'', v_2''v_3''\}$ and $\{e, v_2'v_2'', v_1'v_1'', v_3v_3''\}$ are three perfect matchings in $G'$ that contain $v_1'v_1''$. In both cases, we can obtain that $G'$ contains a $3^+$-matching edge. Hence the result holds.
\end{proof}

\begin{lemma}\label{no-k-tri}
Let $(A,B)$ be the bipartition of $K_{3,3}$, where $A=\{a_1, a_2, a_3\}$, $B=\{b_1, b_2, b_3\}$. Let $T_4=K_{3,3}(a_3)\odot K_4$, and let $v_1$, $v_2$ and $v_3$ be the vertices in the triangle replacement of $T_4$. If $G''$ is a graph obtained from $T_4$ by performing a triangle replacement at a vertex in set $\{v_1, v_2, v_3, b_1, b_2, b_3\}$, then $G''$ contains a $3^+$-matching edge.
 \end{lemma}

\begin{figure}[h!]
  \centering
  \includegraphics[width=5 in]{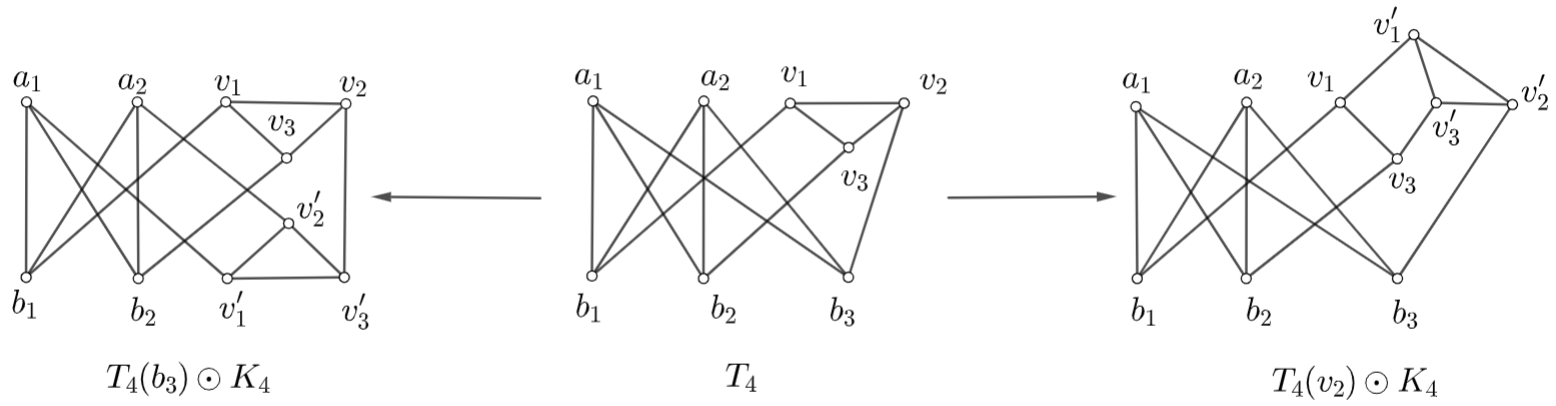}
   \vspace{-0.5cm}
  \caption{Performing a triangle replacement in $T_{4}$}\label{fig14}
\end{figure}

\begin{proof}
Let $v_1'$, $v_2'$ and $v_3'$ be the vertices in the triangle replacement of $G''$. By symmetric, $G''$ is isomorphic to $T_4(v_2)\odot K_4$ or $T_4(b_3)\odot K_4$, see Figure \ref{fig14}. If $G''\cong T_4(v_2)\odot K_4$, then $\{v_1v_1', v_2'v_3', b_2v_3, a_1b_3, a_2b_1\}$, $\{v_1v_1', v_2'v_3', b_2v_3, a_2b_3, a_1b_1\}$ and $\{a_1b_1, a_2b_2, v_1v_1', v_3v_3', b_3v_2'\}$ are three perfect matchings in $G''$ that contain $v_1v_1'$. Hence $v_1v_1'$ is a $3^+$-matching edge in $G''$. If $G''\cong T_4(b_3)\odot K_4$, then $\{a_1b_1, a_2b_2, v_1v_3, v_2v_3', v_1'v_2'\}$, $\{a_1b_2, a_2b_1, v_1v_3, v_2v_3', v_1'v_2'\}$ and $\{a_1v_1', a_2v_2', v_2v_3', b_1v_1, b_2v_3\}$ are three perfect matchings in $G''$ that contain $v_2v_3'$. So $v_2v_3'$ is a $3^+$-matching edge in $G''$.
\end{proof}

Through the same analysis in Lemma \ref{no-k-tri}, we can achieve the following result.

\begin{lemma}\label{no-p-tri}
The graph obtained from the Petersen graph by performing a triangle replacement at a vertex contains a $3^+$-matching edge.
\end{lemma}

%
%

{\bf Proof of Theorem \ref{3-regular}.} The sufficiency is evident since $C_{\bigtriangleup}^1$, $T_2$, $T_4$, $T_5$ or $T_6$ are nonbipartite matching covered graphs with nontrivial cuts in $\mathfrak{B}$. We now prove the necessity. Since $G$ is $3$-regular in $\mathfrak{B}$, every graph produced in the process of tight cuts decomposition of $G$ is also $3$-regular in $\mathfrak{B}$ by Lemmas \ref{B-regular} and \ref{tight-cut&B}. Let $G'$ be the matching covered graph obtained in the process of tight cuts decomposition of $G$ such that $G'$ has a nontrivial tight cut $C$, and the two $C$-contractions $H_1$ and $H_2$ are either bricks or braces. Thus $G'\in \mathfrak{B}$.  Since $C$ is a nontrivial tight cut, $H_1$ and $H_2$ can not be $K_{2}^*$.  Let $\mathfrak{B}_{1}$ be the set of $K_{4}$ and  the Petersen graph (see Figure \ref{T14}). Denote by $\mathfrak{B}_{2}$ the set of $C_{4}^2$ and $K_{3,3}$ (see Figure \ref{T6}). Then by Lemmas \ref{brace-B} and \ref{brick-B}, $H_1$ and $H_2$ are in $\mathfrak{B}_{1}$ or $\mathfrak{B}_{2}$. Since $K_4$ is the only graph that there exists a vertex such that the number of $2$-matching edges incident with it is at most half, we can achieve the following important claim by Lemma \ref{C-DMC}, which used in the remaining part of this proof.

{\bf Claim 1.} Let $H$ be the matching covered graph obtained in the process of tight cuts decomposition of $G$ such that $H$ has a nontrivial tight cut $C$, $H_1$ and $H_2$ are the two $C$-contractions of $H$. If the numbers of $2$-matching edges incident with vertices in $H_1$ are more than half, then $H_2$ is $K_4$.

Now we discuss three cases according to the distribution of $H_1$ and $H_2$.

{\bf Case 1.} $H_1, H_2\in \mathfrak{B}_{2}$. One can check that the numbers of $2$-matching edges incident with vertices in $C_4^2$ and $K_{3,3}$ are more than half. Hence by Lemma \ref{C-DMC}, $G'\notin \mathfrak{B}$, a contradiction.

{\bf Case 2.} $H_1, H_2\in \mathfrak{B}_{1}$. Note that the numbers of $2$-matching edges incident with vertices in the Petersen graph are more than half. Then by Lemma \ref{C-DMC}, at least one of $H_1$ and $H_2$, say $H_1$, is not the Petersen graph. Otherwise, $G'\notin \mathfrak{B}$, a contradiction. Thus $H_1$ is $K_4$ and $|C|=3$. The graph $H_2$ is isomorphic to $K_4$ or the Petersen graph. When $H_2$ is $K_{4}$, then $G'$ is $\overline{C_{6}}$. However, $\overline{C_{6}}$ is not in $\mathfrak{B}$, a contradiction. When $H_2$ is the Petersen graph, then by Lemma \ref{no-p-tri}, $G'$ contains a $3^+$-matching, so $G'$ is not in $\mathfrak{B}$, a contradiction.

{\bf Case 3.}  $H_{1}\in \mathfrak{B}_1$, $H_{2}\in \mathfrak{B}_2$. Then $H_2$ is isomorphic to $C_4^2$ or $K_{3,3}$, and the numbers of $2$-matching edges incident with vertices in $H_2$ are more than half. Then by Claim 1, we have $H_1$ is $K_4$.

{\bf Case 3.1.} If $H_2$ is $C_4^2$, then $G'$ is $C_{\bigtriangleup}^1$ (see Figure \ref{T15}) and is in $\mathfrak{B}$. Then $G$ may be $C_{\bigtriangleup}^1$. When $G\ncong C_{\bigtriangleup}^1$. Let $G''$ be the matching covered graph obtained in the process of tight cuts decomposition of $G$ such that $G''$ has a nontrivial tight cut $C_1$ and $C_{\bigtriangleup}^1$ is one of the two $C_1$-contractions of $G''$. Since the numbers of $2$-matching edges incident with vertices in $C_{\bigtriangleup}^1$ are more than half, the other $C_1$-contraction must be $K_4$ by Claim 1.  Then the graph $G''$ can be viewed as obtained from $C_{\bigtriangleup}^1$ by performing a triangle replacement. Applying for Lemma \ref{no-c-tri}, $G''$ can be obtained from $C_{\bigtriangleup}^1$ by performing a triangle replacement at a vertex that is the end of the multiple edges, say $v_1$. One can check that $G''=C_{\bigtriangleup}^1(v_1)\odot K_4$ and is in $\mathfrak{B}$. Hence $G$ may be $C_{\bigtriangleup}^1(v_1)\odot K_4$. When $G\ncong C_{\bigtriangleup}^1(v_1)\odot K_4$. Let $G'''$ be the matching covered graph obtained in the process of tight cuts decomposition of $G$ such that $G'''$ has a nontrivial tight cut $C_2$, and $C_{\bigtriangleup}^1(v_1)\odot K_4$ is one of the two $C_2$-contractions of $G'''$. Since the numbers of $2$-matching edges incident with vertices in $C_{\bigtriangleup}^1(v_1)\odot K_4$ are more than half, the other $C_2$-contraction must be $K_4$ by Claim 1. Then the graph $G'''$ can be viewed as obtained from $C_{\bigtriangleup}^1(v_1)\odot K_4$ by performing a triangle replacement. However, it will always generate a $3^+$-matching edge in $G'''$ no matter which vertex is performed a triangle replacement on $C_{\bigtriangleup}^1(v_1)\odot K_4$, a contradiction.

{\bf Case 3.2.} If $H_2$ is $K_{3,3}$, then $G'$ is $T_4$ (see Figure \ref{fig14}) and is in $\mathfrak{B}$. Hence $G$ may be $T_4$. When $G\neq T_4$. Let $H''$ be the matching covered graph obtained in the process of tight cuts decomposition of $G$ such that $H''$ has a nontrivial tight cut $C_3$, and $T_4$ is one of the two $C_3$-contractions of $H''$. Since the numbers of $2$-matching edges incident with vertices in $T_4$ are more than half, the other $C_2$-contraction must be $K_4$ by Claim 1. Then the graph $G''$ can be viewed as obtained from $T_4$ by performing a triangle replacement. Hence by the same analysis as that in Case 3.1, we have $G$ is $T_5$ or $T_6$.

In words, if $G$ is nonbipartite $3$-regular matching covered graph with a nontrivial cut in $\mathfrak{B}$, then $G$ is $C_{\bigtriangleup}^1$, $T_2$, $T_4$, $T_5$ or $T_6$. The proof of Theorem \ref{3-regular} is thus completed.

Applying for the same method of analysis, we give the proof of Theorem \ref{4-regular}.

{\bf Proof of Theorem \ref{4-regular}.} The sufficiency is evident since $T_3$ is nonbipartite matching covered graphs with nontrivial cuts in $\mathfrak{B}$. We now prove the necessity. Since $G$ is $4$-regular in $\mathfrak{B}$, every graph produced in the process of tight cuts decomposition of $G$ is also $4$-regular in $\mathfrak{B}$ by Lemmas \ref{B-regular} and \ref{tight-cut&B}. Let $G'$ be the matching covered graph obtained in the process of tight cuts decomposition of $G$ such that $G'$ has a nontrivial tight cut $C$, and the two $C$-contractions $H_1$ and $H_2$ are either bricks or braces. Thus $G'\in \mathfrak{B}$.  Since $C$ is a nontrivial tight cut, $H_1$ and $H_2$ can not be $K_{2}^*$.  Let $\mathfrak{B}_{1}$ be the set of $K_{4}^2$ and $T_1$ (see Figure \ref{T14}). Denote by $\mathfrak{B}_{2}$ the set of $C_{4}^4$ (see Figure \ref{T6}). Then by Lemmas \ref{brace-B} and \ref{brick-B}, $H_1$ and $H_2$ are in $\mathfrak{B}_{1}$ or $\mathfrak{B}_{2}$. Since $K_{4}^2$ is the only graph that there exists a vertex such that the number of $2$-matching edges incident with it is at most half, we can achieve the following claim by Lemma \ref{C-DMC}.

{\bf Claim 2.} Let $H$ be the matching covered graph obtained in the process of tight cuts decomposition of $G$ such that $H$ has a nontrivial tight cut $C$, $H_1$ and $H_2$ are the two $C$-contractions of $H$. If the numbers of $2$-matching edges incident with vertices in $H_1$ are more than half, then $H_2$ is $K_{4}^2$.

Note that $C_{4}^4$ is the only graph in $\mathfrak{B}_{2}$ and the numbers of $2$-matching edges incident with vertices in $C_{4}^4$ are more than half. Then by Lemma \ref{C-DMC}, $H_1$ and $H_2$ can not be both in $\mathfrak{B}_{2}$. With respect to the distribution of $H_1$ and $H_2$, we will consider the following two cases.

{\bf Case 1.} $H_{1}, H_{2}\in \mathfrak{B}_{1}$. Note that the numbers of $2$-matching edges incident with vertices in $T_1$ are more than half. Then by Claim 2, at least one of $H_1$ and $H_2$, say $H_1$, is not $T_1$. Thus $H_1$ is $K_4^2$ and $|C|=4$. The graph $H_2$ is isomorphic to $K_4^2$ or $T_1$. When $H_2$ is $K_4^2$, then $G'$ is $T_3$ and is in $\mathfrak{B}$. Then $G$ may be $T_3$. When $G\neq T_3$. Let $G''$ be the matching covered graph obtained in the process of tight cuts decomposition of $G$ such that $G''$ has a nontrivial tight cut $C_1$ and $T_3$ is one of the two $C_1$-contractions of $G''$. Note that $T_3$ is matching double covered and so the numbers of $2$-matching edges incident with vertices in $T_3$ are more than half. Hence the other $C_1$-contraction must be $K_4^2$ by Claim 2. Then one can check that $G''$ obtained from $K_4^2$ and $K_4^2$ by reverse process of tight cuts decomposition has $3^+$-matching edges, a contradiction.

When $H_2$ is $T_1$. Since $T_1$ is matching double covered and $H_1$ is not $K_{4}$ or an even cycle, the graph $G'$ is not in $\mathfrak{B}$ by Lemma \ref{C-DMC} (2).  is $T_3$ and is in $\mathfrak{B}$, a contradiction.

{\bf Case 2.} $H_{1}\in \mathfrak{B}_1$, $H_{2}\in \mathfrak{B}_2$. Thus $H_2$ is $C_4^4$. Since $C_4^4$ is matching double covered and the numbers of $2$-matching edges incident with vertices in $C_4^4$ are more than half. Then by Claim 2, we have $H_1$ is $K_4^2$. Thus we can check that $G'$ obtained from $C_4^4$ and $K_4^2$ by reverse process of tight cuts decomposition has $3^+$-matching edges, a contradiction.

In both cases, we obtained that $G$ is $T_3$ if $G$ is a nonbipartite $4$-regular matching covered graph with a nontrivial cut in $\mathfrak{B}$. The proof is thus complete. \hfill $\Box$

\section*{Acknowledgement}

The authors thank the referees for their valuable comments. This work is  supported by the National Natural Science Foundation of China under grant numbers 12171440, 12201574 and 12371361.

\end{document}